\documentclass[11pt]{article}
\usepackage[utf8]{inputenc}
\usepackage{amsmath,latexsym,amssymb,amsfonts,amsbsy,amsthm}
\usepackage{graphicx}
\usepackage{subfigure}
\usepackage{algorithm, algorithmicx, algpseudocode}
\usepackage[top=1in, bottom=1in, left=1.25in, right=1.25in]{geometry}
\usepackage{epstopdf}

\usepackage{tabularx}

\usepackage{booktabs}

\usepackage{multirow}

\usepackage{dcolumn}
\newcolumntype{d}[1]{D{.}{.}{#1}}

\usepackage{longtable}

\usepackage{threeparttable}

\usepackage{threeparttablex}

\usepackage{diagbox}

\usepackage{natbib}

\newfont{\bb}{msbm10}

\newtheorem{example}{Example}[section]

\newtheorem{thm}{Theorem}[section]

\newtheorem{remark}{Remark}[section]

\newcommand{\Frac}{\displaystyle\frac}

\newcommand{\iu}{{\rm i}}
\usepackage{color}

\makeatletter 

\@addtoreset{equation}{section}
\makeatother 
\begin{document}
\cleardoublepage \pagestyle{myheadings}

\bibliographystyle{plain}

\title{Lopsided HSS Iterative Method and Preconditioner for a class of Complex Symmetric Linear System}

\author{Yusong Zhang\\
{\it School of Mathematical Sciences}\\
{\it Shanghai Jiao Tong University}\\
{\it 800 Dongchuan Road}\\
{\it Shanghai 200240, P.R.China}\\
{\it Email: zhangyusong\_1@sjtu.edu.cn}\\[2mm]
 Zeng-Qi Wang \footnote{This research was supported
  	by National Key Research and Development Program of China, 2020YFA0709803;\\ General program of Shanghai Natural Science Foundation, 23ZR1433900.}\\
{\it School of Mathematical Sciences}\\
{\it and Ministry of Education Key Lab}\\
{\it of Scientific and Engineering
Computing}\\
{\it Shanghai Jiao Tong University}\\
{\it 800 Dongchuan Road}\\
{\it Shanghai 200240, P.R.China}\\
{\it Email: wangzengqi@sjtu.edu.cn } }

\maketitle

\markboth{\small Y. S. Zhang, Z.-Q. Wang} {\small LHSS method for complex symmetric linear system}

\begin{abstract}
In this study, we propose the lopsided HSS (LHSS) iteration method for solving a class of complex symmetric  indefinite systems of linear equations. This method employs an alternating iterative scheme, where each iteration entails solving two systems of equations with symmetric real coefficient matrices. This design is intended to reduce the high computational costs associated with complex arithmetic. Theoretical analysis shows that the upper bound of the convergence rate depends only on the maximum and minimum eigenvalues of the real symmetric matrices, as well as the iteration parameters. When the eigenvalues satisfy certain conditions, the method guarantees convergence for any positive iteration parameter. Building on this insight, we developed the preconditioned lopsided HSS iteration method (PLHSS). Theoretical results demonstrate that PLHSS exhibits superior convergence properties compared to the original method. Additionally, we derived the optimal parameters for the new approaches and  corresponding optimal  convergence rate.  Furthermore, we derive the PLHSS preconditioner on the basis of the iterative method. The eigenvalues of the preconditioned matrix are well-clustered, and the eigenvectors are orthogonal with a specific inner product. Numerical experiments demonstrate the efficiency of the preconditioned GMRES and COCG methods.   LHSS iteration methods and the relevant preconditioners show mesh size independent and parameter-insensitive convergence behavior for the test numerical examples.

\end{abstract}

\section{Introduction}

We consider the complex symmetric  linear system of equations of the form
\begin{align}\label{equ:system}
    Ax=b,~
\end{align}
where the coefficient matrix $A = W+\mathrm{i}T\in\mathbb{C}^{n\times n}$ is a nonsingular complex matrix with $W\in\mathbb{R}^{n\times n}$ being symmetric positive definite and $T\in\mathbb{R}^{n\times n}$ being symmetric indefinite matrix. We assume $T\neq 0$, which implies that  $A $ is non-Hermitian.  As widely recognized, the system (\ref{equ:system}) is prevalent across a multitude of physical application domains, highlighting its fundamental importance in areas such as computational acoustics, electromagnetics (including antennas, microwaves, and optics), seismic imaging, quantum mechanics simulations, and oscillatory heat transfer/diffusion systems \cite{asoustics,electromagnetics,Tomography,diffusion}. This system arises from the discretization of partial differential equations that model wave propagation and time-harmonic vibrations. Key examples include the Helmholtz equations \cite{Helmholtz} and the frequency-domain Maxwell's equations \cite{Maxwell}.

Over the past years, the solution of systems (\ref{equ:system}) and their equivalent variants have motivated diverse classes of iteration methods, which is well-known as matrix splitting iteration methods.
One of the most classical matrix splitting iterative methods, Hermitian/Skew-Hermitian Splitting (HSS) method \cite{HSS},  stands out by explicitly exploiting the Hermitian/skew-Hermitian structure to achieve faster convergence.  The modified HSS (MHSS) method  \cite{MHSS} was proposed to reduce the computational complexity caused by complex arithmetic.  It  was proven to converge when the coefficient matrices the real part $W$ and imaginary part $T$  are symmetric positive definite and symmetric positive semidefinite, respectively.  This idea inspired a series of variant methods, such as  PMHSS \cite{PMHSS} and LPMHSS \cite{LPMHSS} mthods,  all of which have convergence guarantees under the condition of symmetric positive definite  or semi-definite $W$ and $T$. Several algorithms have also been designed to solve the case when $T$ is indefinite, but they impose strict constraints on the indefinite nature of $T$. For example, the condition $-W\prec T\preceq W$ is regarded as a necessary prerequisite in \cite{Xu-generalization,Cao-Ren-variants} and the implicit condition that $\varepsilon W+T$ is a positive definite matrix for some $\varepsilon$ in \cite{VPMHSS}. These conditions cannot be easily implemented in engineering applications, especially for problems with ill-conditioned matrices. However, the positive definiteness assumption, while sufficient for convergence, isn't always required. For instance, study  \cite{PMHSSPDEC} has shown that PMHSS can achieve convergence even with positive a semidefinite $W$ and an indefinite matrix T   when applied to certain problems.  These methods are also widely applied to solving a special class of block two-by-two real linear systems of  equations \cite{PMHSS-NS, ChebPMHSS,TBTAML}. 

Matrix splitting iterative methods can provide high-quality preconditioners for ill-conditioned systems of linear algebraic equations, thereby enhancing the efficiency of Krylov subspace iteration methods. An effective preconditioner not only tightens the clustering of eigenvalues but also improves the condition number of the eigenvectors. The classical Krylov subspace methods for complex symmetric linear systems include the long-recurrence methods, such as generalized minimal residual (GMRES) method \cite{GMRES} and the short-recurrence methods, such as conjugate orthogonal conjugate gradient (COCG) method \cite{COCG}, orthogonal conjugate residual (COCR) method \cite{COCR} and so on.

This study aims to develop a type of lopsided HSS iterative  methods  for the complex symmetric, partially indefinite systems of the form (\ref{equ:system}).  We establish the convergence results for the proposed LHSS iterative method and its preconditioned variants even $T$ is symmetric indefinite. 
The resulting methods overcome the fundamental limitation of existing approaches that require $T$ to be at least positive semidefinite, while maintaining comparable computational complexity of solving real subsystems.  Iteration methods naturally lead to a preconditioner for the complex matrix $A$. The preconditioning process primarily utilizes real matrix operations, demonstrating superior performance in preconditioned Krylov subspace methods.

The organization of the paper is as follows. In Section \ref{sec:2}, the LHSS iteration method is described and its convergence properties are described. In section \ref{sec:3}, we discuss a preconditioned variant of LHSS iteration and its convergence, which is also guaranteed by the relaxation parameter $\alpha$. In Section \ref{sec:4}, We derive the preconditioner corresponding to the preconditioned LHSS iteration method and provide an analysis from the perspective of the eigenvalues and eigenvectors of the preconditioned coefficient matrix. Numerical results are given in Section \ref{sec:5} to show the effectiveness of (preconditioned) LHSS iteration method as well as the corresponding LHSS preconditioner.

\section{Lopsided HSS iteration method}\label{sec:2}
For the system (\ref{equ:system}) composed of a symmetric positive definite matrix $W$ and a symmetric indefinite matrix $T$, we develop the lopsided HSS (LHSS) iterative method. The format of LHSS iteration method is described as follows:
\begin{align}\label{equ:iter}
    \begin{cases}
    \begin{split}
        \left(\alpha I +W\right)&x^{\left(k+\frac{1}{2}\right)} = \left(\alpha I-\mathrm{i}T\right)x^{\left(k\right)} + b\\
        T &x^{\left(k+1\right)} = \mathrm{i}W x^{\left(k+\frac{1}{2}\right)}-\mathrm{i}b
    \end{split},~k=0,1,2,\cdots,
    \end{cases}
\end{align}
where $\alpha>0$ is a fixed positive real number and  $x^{(0)}\in\mathbb{C}^{n}$ is arbitrarily  chosen initial vector. By  eliminating  the intermediate terms $x^{\left(k+\frac{1}{2}\right)}$ , we reformulate the LHSS iterative scheme as
\begin{align*}
    x^{\left(k+1\right)} = G\left(\alpha\right) x^{\left(k\right)} + B\left(\alpha\right)b,~k=0,1,2,\cdots,
\end{align*}
where
\begin{align}\label{equ:itermatrix}
    G\left(\alpha\right) = \mathrm{i}T^{-1}W\left(\alpha I +W\right)^{-1}\left(\alpha I-\mathrm{i}T\right),
\end{align}
and
\begin{align*}
    B\left(\alpha\right)  = \mathrm{i}T^{-1}W(\alpha I+W)^{-1}-\mathrm{i}T.
\end{align*}
$G\left(\alpha\right)$ in (\ref{equ:itermatrix}) is referred to as the iterative matrix.

\begin{algorithm}[htbp]
	\renewcommand{\algorithmicrequire}{\textbf{Input:}}
	\renewcommand{\algorithmicensure}{\textbf{Output:}}
	\caption{Lopsided HSS iteration method}
	\label{alg:LHSS}
	\begin{algorithmic}[1]
		\Require $b$ , $\alpha$, $x^{(0)}$, $Iter$ ; 
		\Ensure $solution$; 

		\For {$i=1,2,\cdots,Iter$} until converge
		\State Compute $h = \left(\alpha I-\mathrm{i}T\right)x^{\left(k\right)} + b$
		\State Solve $ \left(\alpha I +W\right)x^{\left(k+\frac{1}{2}\right)} = h$
		\State Compute $g = \mathrm{i}W x^{\left(k+\frac{1}{2}\right)}-\mathrm{i}b$
		\State Solve $T x^{\left(k+1\right)} = g $
		\EndFor
		\State $solution = x^{(i)}$
		\State \Return $solution$
	\end{algorithmic}
\end{algorithm}

Denoted by  $0<\lambda_{1}\le\lambda_{2}\le\cdots\le\lambda_{n}:=\lambda_{\max}$ and $\mu_{1}\le \mu_{2}\le \cdots\le\mu_{k}<0<\mu_{k+1}\le\cdots\le\mu_{n}$ the eigenvalues of the matrices $W$ and $T$, respectively. Particularly, $\mu_{\min}$ denotes the smallest eigenvalue of  $T$  in magnitude,  i.e., $\mu_{\min}=\min\lbrace |\mu_{k}|,|\mu_{k+1}|\rbrace$. We  analyze the convergent behavior of the proposed LHSS method by the spectral radius of the iterative matrix $G(\alpha)$.
\begin{thm}\label{thm:iter} 
    Let $A=W+\mathrm{i}T$, with symmetric positive definite matrix $W\in\mathbb{R}^{n\times n}$ and symmetric indefinite matrix $T\in\mathbb{R}^{n\times n}$.  $\lambda_{\max}$ and $\mu_{\min}$ are largest and smallest eigenvalues of  $W$ and $T$ in magnitude, respectively.      Then, the  spectral radius of iterative matrix $G(\alpha)$ in  (\ref{equ:itermatrix})   satisfies $\rho\left(G\left(\alpha\right)\right)\le\sigma\left(\alpha\right)$ for any positive $\alpha$, where
    \begin{align}\label{equ:upbound}
        \sigma\left(\alpha\right)=\frac{\lambda_{\max}}{\alpha+\lambda_{\max}} \cdot \frac{\sqrt{\alpha^{2}+\mu_{\min}^{2}}}{\mu_{\min}}.
    \end{align}
    Furthermore, it holds that \\
    (i) If $\lambda_{\max}\le\mu_{\min}$, then $\sigma(\alpha)<1$  for any $\alpha>0$, which means LHSS iteration scheme (\ref{equ:iter}) is convergent unconditionally;\\
    (ii) If $\lambda_{\max}>\mu_{\min}$, then $\sigma(\alpha)<1$ if and only if
    \begin{align}\label{equ:conditionalconvergence}
        \alpha\in\mathcal{D}_{I}:=\left(0,\frac{2\lambda_{\max}\mu_{\min}^{2}}{\lambda_{\max}^{2}-\mu_{\min}^{2}}\right),
    \end{align}
 i.e., the iteration method (\ref{equ:iter}) is convergent if   condition (\ref{equ:conditionalconvergence}) is true.
\begin{proof}[Proof]
Since $W$ and $T$ are symmetric, there exists orthogonal matrices $P$ and $Q$, such that $P^{T}WP=\Lambda_{W},~Q^{T}TQ=\Lambda_{T}$, where
    \begin{align*}
        \Lambda_{W}=\mathrm{diag}\left(\lambda_{1},\lambda_{2},\cdots,\lambda_{n}\right),~
        \Lambda_{T}=\mathrm{diag}\left(\mu_{1},\mu_{2},\cdots,\mu_{n}\right).
    \end{align*}
    It holds that 
    \begin{align*}
        \rho\left(G\left(\alpha\right)\right)
        &=\rho\left(\mathrm{i}T^{-1}W\left(\alpha I +W\right)^{-1}\left(\alpha I-\mathrm{i}T\right)\right)\\
        &=\rho\left(
        W\left(\alpha I +W\right)^{-1}\left(\alpha I-\mathrm{i}T\right)T^{-1}\right)\\
        &\le\left\|
        W\left(\alpha I +W\right)^{-1}\left(\alpha I-\mathrm{i}T\right)T^{-1}\right\|_{2}\\
        &\le\left\|
        W\left(\alpha I +W\right)^{-1}\right\|_{2}\left\|\left(\alpha I-\mathrm{i}T\right)T^{-1}\right\|_{2}\\
        & \le   \left\|
        \Lambda_{W}\left(\alpha I +\Lambda_{W}\right)^{-1}\right\|_{2}\left\|\left(\alpha I-\mathrm{i}\Lambda_{T}\right)\Lambda_{T}^{-1}\right\|_{2}\\
        &=\mathop{\max}\limits_{1\le j\le n}\frac{\lambda_{j}}{\alpha+\lambda_{j}} \cdot \mathop{\max}\limits_{1\le j\le n}\sqrt{\frac{\alpha^{2}+\mu_{j}^{2}}{\mu_{j}^{2}}}.
    \end{align*}
    Furthermore, each component of this estimation formula can be expressed in the following form
    \begin{align*}
    	\mathop{\max}\limits_{1\le j\le n}\frac{\lambda_{j}}{\alpha+\lambda_{j}}=\frac{\lambda_{\max}}{\alpha+\lambda_{\max}},
    \end{align*}
	and
	\begin{align*}
		\mathop{\max}\limits_{1\le j\le n}\sqrt{\frac{\alpha^{2}+\mu_{j}^{2}}{\mu_{j}^{2}}}=\frac{\sqrt{\alpha^{2}+\mu^{2}_{\min}}}{\mu_{\min}},
    \end{align*}
    respectively. The upper bound of $\rho(G\left(\alpha\right))$ given in (\ref{equ:upbound}) is obtained directly.
	
Therefore,  upper bound $\sigma\left(\alpha\right)<1$ is equivalent to
    \begin{align}\label{equ:condition}
        \alpha\left(\lambda_{\max}^{2}-\mu^{2}_{\min}\right)-2\mu^{2}_{\min}\lambda_{\max}<0.
    \end{align}
   If $\lambda_{\max}^{2}\le\mu^{2}_{\min}$, then (\ref{equ:condition}) holds true for any positive constant $\alpha>0$, i.e., the LHSS iteration (\ref{equ:iter}) converges to the unique solution of the system (\ref{equ:system}); if $\lambda_{\max}^{2}>\mu^{2}_{\min}$, then the LHSS iteration (\ref{equ:iter}) converges to the unique solution of the system (\ref{equ:system}) if and only if $\alpha$ satisfies (\ref{equ:conditionalconvergence}). 
\end{proof}
\end{thm}

Theorem \ref{thm:iter} illustrates the range of parameter  that enables the convergence of the LHSS iterative method (\ref{equ:iter}). The following theorem provides a quasi-optimal parameter  by minimizing the upper bound of the spectral radius. 
\begin{thm}\label{thm:optimal}
Under the assumptions of Theorem \ref{thm:iter},  the quasi optimal parameter $\alpha^{*}=\frac{\mu^{2}_{\min}}{\lambda_{\max}}$  minimizes the upper bound $\sigma(\alpha)$ (\ref{equ:upbound}), i.e. 
	\begin{align}\label{equ:optimal-speed}
	\min\limits_{\alpha>0}  \sigma(\alpha)	=\sigma\left(\alpha^{*}\right)=\frac{\lambda_{\max}}{\sqrt{\lambda_{\max}^{2}+\mu_{\min}^{2}}} .
	\end{align}
    \begin{proof}[proof]
        Notice that  $\sigma(\alpha)>0$  is continuous with respect to $\alpha$, and 
    \begin{align*}
        \sigma^{2}(\alpha)=\frac{\lambda_{\max}^{2}}{\mu^{2}_{\min}}\cdot \frac{\alpha^{2}+\mu^{2}_{\min}}{\left(\alpha+\lambda_{\max}\right)^{2}}.
    \end{align*}
    Then we get 
    \begin{align*}
        \frac{\mathrm{d}}{\mathrm{d}\alpha}\sigma^{2}(\alpha)=\frac{2\alpha\lambda_{\max}-2\mu^{2}_{\min}}{\left(\alpha+\lambda_{\max}\right)^{3}}.
    \end{align*}
    The square of upper bound function $\sigma^{2}(\alpha)$ is monotonically decreasing on the interval $\left[0,\frac{\mu^{2}_{\min}}{\lambda_{\max}}\right]$ and monotonically increasing on the interval $\left[\frac{\mu^{2}_{\min}}{\lambda_{\max}},+\infty\right)$. Hence,   $\sigma(\alpha)$  reaches its minimum value at $\alpha^{*}=\frac{\mu^{2}_{\min}}{\lambda_{\max}}$. Then
    \begin{align*}        \rho\left(G\left(\alpha^{*}\right)\right)\le\sigma\left(\alpha^{*}\right)=\frac{\lambda_{\max}}{\sqrt{\mu^{2}_{\min}+\lambda_{\max}^{2}}}<1.
    \end{align*}
    The last inequality holds because of the non-singularity of matrix $T$, from which we obtain $\mu_{\min}\neq 0$. 
    \end{proof}
\end{thm}

\section{The Preconditioned LHSS iteration method}\label{sec:3}
In this section, we  use the preconditioning to  accelerate the original LHSS iterative method by applying an effective preconditioner $V$ to precondition the system (\ref{equ:system}). Let $V\in\mathbb{R}^{n\times n}$ be a symmetric positive definite matrix. Based on the notations
\begin{align*}
    \tilde{W}=V^{-\frac{1}{2}}WV^{-\frac{1}{2}},~\tilde{T}=V^{-\frac{1}{2}}TV^{-\frac{1}{2}},~\tilde{A}=V^{-\frac{1}{2}}AV^{-\frac{1}{2}},
\end{align*}
and
\begin{align*}
    \tilde{x}=V^{\frac{1}{2}}x,~\tilde{b}=V^{-\frac{1}{2}}b,
\end{align*}
we can equally reformulate (\ref{equ:system}) to 
\begin{align}\label{equ:Psystem}
    \tilde{A}\tilde{x}=\left(\tilde{W}+\mathrm{i}\tilde{T}\right)\tilde{x}=\tilde{b}.
\end{align}

Now we are in the position of applying LHSS iteration method (\ref{equ:iter}) to preconditioned system (\ref{equ:Psystem}). In this case, we acquire  for solving the system (\ref{equ:system}) with 
and preconditioned LHSS (PLHSS) iteration method is derived as
\begin{align}\label{equ:Piter}
    \begin{cases}
    \begin{split}
        \left(\alpha V +W\right)&x^{\left(k+\frac{1}{2}\right)} = \left(\alpha V-\mathrm{i}T\right)x^{\left(k\right)} + b\\
        T &x^{\left(k+1\right)} = \mathrm{i}W x^{\left(k+\frac{1}{2}\right)}-\mathrm{i}b,
    \end{split}
    \end{cases}
\end{align}
which is a special case of LHSS iteration method (\ref{equ:iter}) when $V=I$. The PLHSS iteration method could also be rewritten as standard iterative scheme 
\begin{align}
    x^{\left(k+1\right)} = G\left(V;\alpha\right) x^{\left(k\right)} + B\left(V;\alpha\right)b,~k=0,1,2\cdots,
\end{align}
with the notations
\begin{align}
    G\left(V;\alpha\right) = \mathrm{i}T^{-1}W\left(\alpha V +W\right)^{-1}\left(\alpha V-\mathrm{i}T\right),
\end{align}
and
\begin{align}
    B\left(V;\alpha\right)  = \mathrm{i}T^{-1}W(\alpha V+W)^{-1}-\mathrm{i}T.
\end{align}

It is worth noting that if we take the preconditioner as $V=I$, $G\left(V;\alpha\right)$ and $B\left(V;\alpha\right)$ degenerate to $G\left(\alpha\right)$ and $B\left(\alpha\right)$, respectively. It is equivalent to keeping the original problem unpreconditioned.

\begin{algorithm}[htbp]
	\renewcommand{\algorithmicrequire}{\textbf{Input:}}
	\renewcommand{\algorithmicensure}{\textbf{Output:}}
	\caption{Preconditioned LHSS iteration method}
	\label{alg:PreLHSS}
	\begin{algorithmic}[1]
		\Require $b$ , $\alpha$, $x^{(0)}$, $Iter$ ; 
		\Ensure $solution$; 
		
		\For {$i=1,2,\cdots,Iter$} until converge
		\State Compute $h = \left(\alpha V-\mathrm{i}T\right)x^{\left(k\right)} + b$
		\State Solve $\left(\alpha V +W\right)x^{\left(k+\frac{1}{2}\right)} = h$
		\State Compute $g = \mathrm{i}W x^{\left(k+\frac{1}{2}\right)}-\mathrm{i}b$
		\State Solve $T x^{\left(k+1\right)} = g $
		\EndFor
		\State $solution = x^{(i)}$
		\State \Return $solution$
	\end{algorithmic}
\end{algorithm}

Suppose $0<\tilde{\lambda}_{1}\le \tilde{\lambda}_{2}\le \cdots\le \tilde{\lambda}_{n}:=\tilde{\lambda}_{\max}$ represent the eigenvalues of $V^{-1}W$ and $\tilde{\mu}_{1}\le \tilde{\mu}_{2}\le  \cdots\le\tilde{\mu}_{k}<0<\tilde{\mu}_{k+1}\le\cdots\le\tilde{\mu}_{n}$ represent the eigenvalues of $V^{-1}T$, respectively. In particular, $\tilde{\mu}_{\min}$ is the eigenvalue of $V^{-1}T$  whose absolute value is closest to zero. Namely, $\tilde{\mu}_{\min}=\min\lbrace |\tilde{\mu}_{k}|,|\tilde{\mu}_{k+1}|\rbrace$.  It should be noticed that $V^{-1}W$ is in a similar relationship with $\tilde{W}$ and $V^{-1}T$ is similar to $\tilde{T}$ so that they share the same eigenvalues pairwise.

\begin{thm}\label{thm:Piter}
	Let complex symmetric matrix $A=W+\mathrm{i}T$, with symmetric positive definite matrix $W\in\mathbb{R}^{n\times n}$ and symmetric indefinite matrix $T\in\mathbb{R}^{n\times n}$, and  $V\in\mathbb{R}^{n\times n}$ be an symmetric positive definite matrix.  
  $\tilde{\lambda}_{\max}$ and $\tilde{\mu}_{\min}$ are largest and smallest eigenvalues of  $V^{-1}W$ and $V^{-1}T$ in magnitude, respectively.      Then
the spectral radius of iterative matrix corresponding to PLHSS iteration method (\ref{equ:Piter}) satisfies $\rho\left(G\left(V;\alpha\right)\right)\le\tilde{\sigma}\left(\alpha\right) $ for any positive $\alpha$, where
	\begin{align}\label{equ:Pupbound}
		\tilde{\sigma}\left(\alpha\right)=\frac{\tilde{\lambda}_{\max}}{\alpha+\tilde{\lambda}_{\max}} \cdot \frac{\sqrt{\alpha^{2}+\tilde{\mu}_{\min}^{2}}}{\tilde{\mu}_{\min}},
	\end{align}
	Moreover, it can be concluded as follows.\\
    (i) If $\tilde{\lambda}_{\max}\le\tilde{\mu}_{\min}$, then  $\tilde{\sigma}(\alpha)<1$ for any $\alpha>0$. The PLHSS iteration method (\ref{equ:Piter}) is actually unconditionally convergent.\\
    (ii) If $\tilde{\lambda}_{\max}>\tilde{\mu}_{\min}$, then  $\tilde{\sigma}(\alpha)<1$  holds once the condition 
    \begin{align}\label{equ:Pconditionalconvergence}
        \alpha\in\mathcal{D}_{V}:=\left(0,\frac{2\tilde{\lambda}_{\max}\tilde{\mu}_{\min}^{2}}{\tilde{\lambda}_{\max}^{2}-\tilde{\mu}_{\min}^{2}}\right)
    \end{align}
    is satisfied. That is to say, the PLHSS iteration method (\ref{equ:Piter}) is convergent under the condition (\ref{equ:Pconditionalconvergence}).
    \begin{proof}[Proof]
    	It is by using direct calculation method that  the following estimation is obtained
        \begin{align*}
            \rho\left(G\left(V;\alpha\right)\right)&=\rho\left(\mathrm{i}T^{-1}W\left(\alpha V +W\right)^{-1}\left(\alpha V-\mathrm{i}T\right)\right)\\
            &=\rho\left(\mathrm{i}V^{-\frac{1}{2}}\tilde{T}^{-1}\tilde{W}\left(\alpha I +\tilde{W}\right)^{-1}\left(\alpha I-\mathrm{i}\tilde{T}\right)V^{\frac{1}{2}}\right)\\
            &=\rho\left(\mathrm{i}\tilde{W}\left(\alpha I +\tilde{W}\right)^{-1}\left(\alpha I-\mathrm{i}\tilde{T}\right)\tilde{T}^{-1}\right)\\
            &\le\left|\left|\mathrm{i}\tilde{W}\left(\alpha I +\tilde{W}\right)^{-1}\left(\alpha I-\mathrm{i}\tilde{T}\right)\tilde{T}^{-1}\right|\right|_{2}\\
            &\le\left|\left|\tilde{W}\left(\alpha I +\tilde{W}\right)^{-1}\right|\right|_{2}\left|\left|\left(\alpha I-\mathrm{i}\tilde{T}\right)\tilde{T}^{-1}\right|\right|_{2}.
        \end{align*}
	Analogous to the proof of Theorem \ref{thm:iter}, we have
    \begin{align*}
        \rho\left(G\left(V;\alpha\right)\right)\le \mathop{\max}\limits_{1\le j\le n}\frac{\tilde{\lambda}_{j}}{\alpha+\tilde{\lambda}_{j}} \cdot \mathop{\max}\limits_{1\le j\le n}\sqrt{\frac{\alpha^{2}+\tilde{\mu}_{j}^{2}}{\tilde{\mu}_{j}^{2}}}.
    \end{align*}
    Substitute
    \begin{align*}
	    \mathop{\max}\limits_{1\le j\le n}\frac{\tilde{\lambda}_{j}}{\alpha+\tilde{\lambda}_{j}}=\frac{\tilde{\lambda}_{\max}}{\alpha+\tilde{\lambda}_{\max}}
    \end{align*}
    and
    \begin{align*}
    	\mathop{\max}\limits_{1\le j\le n}\sqrt{\frac{\alpha^{2}+\tilde{\mu}_{j}^{2}}{\tilde{\mu}_{j}^{2}}}=\frac{\sqrt{\alpha^{2}+\tilde{\mu}_{\min}^{2}}}{\tilde{\mu}_{\min}}
    \end{align*}
    into the above inequality, hence the spectral radius of the iteration matrix $G\left(V;\alpha\right)$ is bounded by
    \begin{align*}
        \rho(G\left(V;\alpha\right))&\le\left|\left|\tilde{W}\left(\alpha I +\tilde{W}\right)^{-1}\right|\right|_{2}\left|\left|\left(\alpha I-\mathrm{i}\tilde{T}\right)\tilde{T}^{-1}\right|\right|_{2}\\
        &\le \frac{\tilde{\lambda}_{\max}}{\alpha+\tilde{\lambda}_{\max}} \cdot \frac{\sqrt{\alpha^{2}+\tilde{\mu}^{2}_{\min}}}{\tilde{\mu}_{\min}}:=\tilde{\sigma}(\alpha),
    \end{align*}
    which leads to (\ref{equ:Pupbound}).

    Next, consider the equivalent condition for $\tilde{\sigma}\left(\alpha\right)<1$, i.e.,
    \begin{align}\label{equ:Pcondition}
        \alpha\left(\tilde{\lambda}_{\max}^{2}-\tilde{\mu}_{\min}^{2}\right)-2\tilde{\mu}_{\min}^{2}\tilde{\lambda}_{\max}<0.
    \end{align}
  Evidently, the condition (\ref{equ:Pcondition}) is true for any positive number $\alpha>0$ provided $\tilde{\lambda}_{\max}^{2}\le\tilde{\mu}_{min}^{2}$ and we conclude the PLHSS iteration method (\ref{equ:Piter}) is unconditionly convergent to the unique solution of (\ref{equ:system}) correspondingly. If the given eigenvalues meet the relationship $\tilde{\lambda}_{\max}^{2}>\tilde{\mu}_{min}^{2}$, then we reach the conclusion that the iteration method (\ref{equ:Piter}) converges to the solution of (\ref{equ:system}) if and only if the parameter $\alpha$ complies with (\ref{equ:Pconditionalconvergence}).
    \end{proof}
\end{thm}
Consequently, we suggest  an quasi optimal parameter $\alpha$ and the corresponding convergence rate with respect to special choice of the preconditioner $V$.
\begin{thm}\label{thm:Poptimal}
	Under the assumption of Theorem \ref{thm:Piter}, the optimal parameter $\tilde{\alpha}^{*}=\frac{\tilde{\mu}^{2}_{\min}}{\tilde{\lambda}_{\max}}$  minimizes the upper bound (\ref{equ:Pupbound}), and the PLHSS iterative scheme (\ref{equ:Piter}) with $\alpha^{*}$ converges to the solution of (\ref{equ:system}).
    \end{thm}
    \begin{proof}[proof]
        The proof is the same as that of Theorem \ref{thm:optimal}.
    \end{proof}

 There is a nature choice of $V = W$. We analysis the iterative method with the specific cases. The iterative scheme    (\ref{equ:Piter}) is reorganized  as
\begin{equation}\label{equ:PiterW}
 T x^{\left(k+1\right)} = \Frac{1}{\alpha +1} (\iu \alpha W+T)x^{(k)} -\frac{\alpha \iu}{\alpha +1}b.
\end{equation}
    The iterative matrix is 
    \begin{align}\label{equ:PLHSSWmatrix}
        G\left(W;\alpha\right) = \frac{1}{\alpha+1}\left(\mathrm{i}\alpha T^{-1}W+I\right).
    \end{align}
 This scheme significantly reduces the computational workload by halving the effort required to solve a linear system at each iteration. 
    \begin{thm}
Let $A=W+\mathrm{i}T$ with symmetric positive definite matrix $W\in\mathbb{R}^{n\times n}$ and symmetric indefinite matrix $T\in\mathbb{R}^{n\times n}$. $\xi_{\max}$ is the largest eigenvalue of $T^{-1}W$ in magnitude. Then the spectral radius of iteration matrix $G\left(W;\alpha\right)$ in (\ref{equ:PLHSSWmatrix}) satisfies 
    \begin{align*}
        \rho\left(G\left(W;\alpha\right)\right)= \frac{1}{\alpha+1}\sqrt{\alpha^{2}\xi_{\max}^{2}+1}.
    \end{align*}
    In addition, \\
    (i) If $\left|\xi_{\max}\right|\le 1$, then $\rho\left(G\left(W;\alpha\right)\right)<1$ holds for any $\alpha>0$. The  iterative scheme (\ref{equ:PiterW}) method  converges unconditionally.\\
    (ii) If $\left|\xi_{\max}\right|> 1$, then $\rho\left(G\left(W;\alpha\right)\right)<1$ provided
    \begin{align}\label{equ:Wcondition}
        \alpha \in \mathcal{D}_{W}:=\left(0,\frac{2}{\xi_{\max}^{2}-1}\right).
    \end{align}    \end{thm}

\begin{remark} \label{rmk:optalpW}
The optimal parameter $\alpha^{*}_{W}=\xi_{\max}^{-2}$ minimizes the spectral radius $\rho\big(G(W;\alpha)\big)$. We notice
    \[
   | \xi_{\max}| =\mathop{max}_{\lambda\in sp(T^{-1}W)}|\lambda|\le \left|\frac{\lambda_{\max}}{\mu_{\min}} \right|, 
    \]
therefore,
	\begin{align*}
		\rho\big(G(W;\alpha^{*}_{W})\big)=\frac{1}{\sqrt{\xi_{\max}^{-2}+1}}\le\frac{1}{\sqrt{\frac{\mu_{\min}^{2}}{\lambda_{\max}^{2}}+1}}=\frac{\lambda_{\max}}{\sqrt{\lambda_{\max}^{2}+\mu_{\min}^{2}}}=\sigma\left(\alpha^{*}\right).
	\end{align*}
It indicates that the preconditioned LHSS iterative scheme (\ref{equ:PiterW}) is more efficient than LHSS iterative scheme (\ref{equ:iter}) they adopt their respective optimal parameters.
Furthermore,   it holds that
	\begin{align*}
		\frac{2\lambda_{\max}\mu_{\min}^{2}}{\lambda_{\max}^{2}-\mu_{\min}^{2}}\le \frac{2}{\xi_{\max}^{2}-1}, \quad \mbox{when} \quad  \lambda_{\max}\le 1 ,
	\end{align*}
thus $\mathcal{D}_{I}\subseteq\mathcal{D}_{W}$.   It indicates that  preconditioned LHSS iterative scheme has larger convergence domain in this case.
\end{remark}

The PLHSS iteration method (\ref{equ:Piter}) is well defined also when $V = T$,  and can be  rewritten as
\begin{align}\label{equ:PiterT}
\begin{cases}
    \begin{split}
        \left(\alpha T +W\right)&x^{\left(k+\frac{1}{2}\right)} = \left(\alpha-\mathrm{i}\right)Tx^{\left(k\right)} + b\\
        T &x^{\left(k+1\right)} = \mathrm{i}W x^{\left(k+\frac{1}{2}\right)}-\mathrm{i}b.
    \end{split}
    \end{cases}
\end{align}
The iterative matrix is 
\begin{align}\label{equ:PLHSSTmatrix}
G\left(T;\alpha\right) = \left(1+\mathrm{i}\alpha\right)T^{-1}W\left(\alpha T+W\right)^{-1}T.
\end{align}
Denote the eigenvalues of matrix $T^{-1}W$ by $\xi_{i}$,
\[
\xi_{\max}^{-} = \mathrm{min}\lbrace\xi_{i}<0 \rbrace,  \qquad \xi_{\max}^{+} = \mathrm{max}\lbrace\xi_{i}>0\rbrace, \qquad i=1,2,\cdots,n
\]
 Then the eigenvalues of matrix $G\left(T;\alpha\right)$ can be expressed as 
\begin{align*}
	\lambda\left(G\left(T;\alpha\right)\right) = (1+\mathrm{i}\alpha)\left(\frac{\alpha}{\xi_{i}}+1\right)^{-1},~i=1,2,\cdots,n .
\end{align*}
 The convergence behavior is analyzed as follows.
\begin{thm}\label{thm:iterT}
Let $A=W+\mathrm{i}T$ with sysmetric positive definite matrix $W\in\mathbb{R}^{n\times n}$ and symmetric indefinite matrix $T\in\mathbb{R}^{n\times n}$.  Then the spectral radius of iteration matrix (\ref{equ:PLHSSTmatrix}) corresponding to PLHSS iteration method (\ref{equ:PiterT}) could be written as
    \begin{equation} \label{eqn:rho_GT}
    \rho\left(G\left(T;\alpha\right)\right)=\sqrt{1+\alpha^{2}}\max\limits _{1\leq i\leq n}\frac{\left|\xi_{i}\right|}{\left|\alpha+\xi_{i}\right|}.
            \end{equation}      
    Suppose that $\xi_{\max}^{-}\in(-1,0)$, it holds that\\  
    (i) If $\xi_{\max}^{+}\le 1$, then $\rho\left(G\left(T;\alpha\right)\right)<1$ provided 
    \begin{align*}
        \alpha\in\left(\frac{-2\xi_{\max}^{-}}{1-(\xi_{\max}^{-})^{2}},+\infty\right);
    \end{align*}\\
    (ii) If $\xi_{\max}^{+}> 1$, then $\rho\left(G\left(T;\alpha\right)\right)<1$ as long as  $\xi_{\max}^{-}\xi_{\max}^{+}>-1$ and
    \begin{align*}
        \alpha\in\left(\frac{-2\xi_{\max}^{-}}{1-(\xi_{\max}^{-})^{2}},\frac{2\xi_{\max}^{+}}{(\xi_{\max}^{+})^{2}-1}\right).
    \end{align*}
    \begin{proof}
        To ensure
        \begin{align*}
            \rho\left(G\left(T;\alpha\right)\right)=\sqrt{1+\alpha^{2}}\max\limits _{1\leq i\leq n}\frac{\left|\xi_{i}\right|}{\left|\alpha+\xi_{i}\right|}<1,
        \end{align*}
    it must hold that  
    \begin{equation*}
  (1+\alpha^{2}) \xi_{i}^{2} < (\alpha+\xi_{i})^{2}, \quad \mbox{for any}  ~\xi_{i},
          \end{equation*}
 which is equivalent to
\begin{equation}\label{alphaxi}
\alpha(\xi_{i}^2-1)<2\xi_{i}.    
\end{equation}
For any $\xi_i \le 1$, (\ref{alphaxi}) holds true. For any $\xi_{i}> 1$,  (\ref{alphaxi}) indicates that 
        \[
        \alpha<\dfrac{2\xi_{\max}^{+}}{(\xi_{\max}^{+})^{2}-1}\leq \dfrac{2\xi_{i}}{(\xi_{i})^{2}-1}.
        \] 
For any $-1<\xi_i < 0$, it also holds that $1 - \xi_i^2 < 0$. Inequality (\ref{alphaxi}) indicates that 
        \[
            \alpha >  \frac{-2\xi_{\max}^{-}}{1-(\xi_{\max}^{-})^{2}} \geq \frac{-2\xi_i}{1- \xi_i^2} . 
        \]
        Therefore, we discuss the following cases:
 \begin{itemize}
        \item[(i)] When $\xi_{\max}^{+} \le 1$,
        \[
            \alpha >   \frac{-2\xi_{\max}^{-}}{1-(\xi_{\max}^{-})^{2}};
        \]
        \item[(ii)]  When $\xi_{\max}^{+}>1$,   $\alpha$ should satisfies
        \[
            \alpha > \frac{2\xi_i}{ \xi_i^2-1} \text{ for } -1<\xi_i < 0, \quad \mbox{and } \quad  \alpha<\frac{2\xi_i}{\xi_i^2-1} \text{ for } \xi_i
            >1.
        \]
        It equivalent that
        \[
        \alpha<\dfrac{2\xi_{\max}^{+}}{(\xi_{\max}^{+})^{2}-1}, \quad \mbox{and} \quad   \alpha >   \frac{-2\xi_{\max}^{-}}{1-(\xi_{\max}^{-})^{2}}.
        \] 
        The  $\alpha$ exists  if $\xi_{\max}^{-}\xi_{\max}^{+}>-1$.
            \end{itemize}
    \end{proof}
\end{thm}
  \begin{thm}\label{thm:optalphaT}
Denoted by $\Theta = \left(\xi_{\max}^{-}\right)^{-1}+\left(\xi_{\max}^{+}\right)^{-1}$. Based on the hypothesis in Theorem \ref{thm:iterT}, the optimal parameter 
\begin{align*}
    \alpha^{*}_{T}=
    \begin{cases}
        +\infty,& ~\Theta\ge 0,\\
        \mathrm{max}\lbrace \left(\xi_{\max}^{+}\right)^{-1},~-2\Theta^{-1}\rbrace,&~\Theta<0,
    \end{cases}
\end{align*}
minimizes the spectral radius $\rho\left(G\left(T;\alpha\right)\right)$ in (\ref{eqn:rho_GT}), and the PLHSS iteration scheme (\ref{equ:PiterT}) with $\alpha^{*}_{T}$ converges to the solution of (\ref{equ:system}).
\end{thm}
\begin{proof}
We notice that
\begin{align*}
    \alpha>\dfrac{-2\xi_{\max}^{-}}{1-(\xi_{\max}^{-})^{2}},~-1<\xi_{\max}^{-}<0
\end{align*}
is a necessary condition to promise the convergence of PLHSS iteration method with $V=T$ based on Theorem \ref{thm:iterT}. In this situation, we have
\begin{align*}
        -\alpha<\dfrac{2\xi_{\max}^{-}}{1-(\xi_{\max}^{-})^{2}}<\xi_{\max}^{-}\le\xi_{i}\le\xi_{\max}^{+},~i=1,2,\cdots,n,
    \end{align*}
    which implies
    \begin{align*}
        \dfrac{|\xi_{i}|}{\left|\alpha+\xi_{i}\right|}\le\dfrac{|\xi_{\max}^{-}|}{\left|\alpha+\xi_{\max}^{-}\right|} \quad \mathrm{or}\quad \dfrac{|\xi_{i}|}{\left|\alpha+\xi_{i}\right|}\le\dfrac{|\xi_{\max}^{+}|}{\left|\alpha+\xi_{\max}^{+}\right|}:=\Xi_{2}\left(\alpha\right),~i=1,2,\cdots,n.
    \end{align*}
    Therefore, the maximum value of $\sqrt{1+\alpha^{2}}\dfrac{|\xi_{i}|}{\left|\alpha+\xi_{i}\right|}$ is attained either at $\xi_{i}=\xi_{\max}^{-}$ or at $\xi_{i}=\xi_{\max}^{+}$. Furthermore, the following inequality holds true:
    \begin{align*}
        \dfrac{|\xi_{\max}^{-}|}{\left|\alpha+\xi_{\max}^{-}\right|}  >\dfrac{|\xi_{\max}^{+}|}{\left|\alpha+\xi_{\max}^{+}\right|},~\alpha\le-2\Theta^{-1},\\
        \dfrac{|\xi_{\max}^{-}|}{\left|\alpha+\xi_{\max}^{-}\right|}< \dfrac{|\xi_{\max}^{+}|}{\left|\alpha+\xi_{\max}^{+}\right|},~\alpha>-2\Theta^{-1}.
    \end{align*}
     Denote two functions on $(0,+\infty)$ as 
    \begin{align*}
    \Xi_{1}\left(\alpha\right):=\sqrt{1+\alpha^{2}}\dfrac{|\xi_{\max}^{-}|}{\left|\alpha+\xi_{\max}^{-}\right|}=\sqrt{1+\alpha^{2}}\dfrac{-\xi_{\max}^{-}}{\alpha+\xi_{\max}^{-}},\\
    \Xi_{2}\left(\alpha\right):=\sqrt{1+\alpha^{2}}\dfrac{|\xi_{\max}^{+}|}{\left|\alpha+\xi_{\max}^{+}\right|}=\sqrt{1+\alpha^{2}}\dfrac{\xi_{\max}^{+}}{\alpha+\xi_{\max}^{+}}.
    \end{align*}
    By calculating the derivative with respect to $\alpha$ gives
    \begin{align*}
        \frac{\mathrm{d} \Xi_{1}\left(\alpha\right)}{\mathrm{d}\alpha} = \frac{-\xi_{\max}^{-}\left(\alpha\xi_{\max}^{-}-1\right)}{\sqrt{1+\alpha^{2}}\left(\alpha+\xi_{\max}^{-}\right)^{2}},\quad \mbox{and}\quad 
        \frac{\mathrm{d} \Xi_{2}\left(\alpha\right)}{\mathrm{d}\alpha} = \frac{\xi_{\max}^{+}\left(\alpha\xi_{\max}^{+}-1\right)}{\sqrt{1+\alpha^{2}}\left(\alpha+\xi_{\max}^{+}\right)^{2}},
    \end{align*}
    which indicates $\Xi_{1}\left(\alpha\right)$ is monotonically decreasing for $\alpha>0$, while the function $\Xi_{2}\left(\alpha\right)$ is monotonically decreasing in the interval $\left(0,\left(\xi_{\max}^{+}\right)^{-1}\right)$ and monotonically increasing in the interval $\left(\left(\xi_{\max}^{+}\right)^{-1},+\infty\right)$.
    \begin{itemize}
        \item[(i)]   If $\Theta\ge 0$, then $\alpha>0>-2\Theta^{-1}$. Clearly this situation only occurs when $\xi_{\max}^{+}\le 1$. The spectral radius could be written explicitly as 
    \begin{align*}
        \rho\left(G\left(T;\alpha\right)\right) = \Xi_{1}\left(\alpha\right),~\alpha>\dfrac{-2\xi_{\max}^{-}}{1-(\xi_{\max}^{-})^{2}},
    \end{align*}
    which is monotonically decreasing within its domain. Therefore, $\rho\left(G\left(T;\alpha\right)\right)$ reaches its minimum value at $\alpha^{*}_{T}=+\infty$.
        \item[(ii)]  If $\Theta < 0$, we analyze $\rho\big(G(T;\alpha)\big)$ by considering two cases: $\xi_{\max}^{+} \le 1$ and $\xi_{\max}^{+} > 1$. Theorem \ref{thm:iterT} implies distinct ranges of $\alpha$ for each case.
        \begin{itemize}
            \item[(a)] If $\xi_{\max}^{+}\le 1$, then $\alpha>\dfrac{-2\xi_{\max}^{-}}{1-(\xi_{\max}^{-})^{2}}$. The spectral radius is
            \begin{align*}
        \rho\left(G\left(T;\alpha\right)\right) = 
        \begin{cases}
            \Xi_{1}\left(\alpha\right),~\dfrac{-2\xi_{\max}^{-}}{1-(\xi_{\max}^{-})^{2}}<\alpha\le-2\Theta^{-1},\\[10pt]
            \Xi_{2}\left(\alpha\right),~\alpha>-2\Theta^{-1}.
        \end{cases}
    \end{align*}
    In terms of the monotonicity of $\Xi_{1},\Xi_{2}$, we conclude $\alpha=-2\Theta^{-1}$ minimizes $\rho\left(G\left(T;\alpha\right)\right)$ on $\left(\dfrac{-2\xi_{\max}^{-}}{1-(\xi_{\max}^{-})^{2}},-2\Theta^{-1}\right]$, while $\alpha=\mathrm{max}\lbrace \left(\xi_{\max}^{+}\right)^{-1},~-2\Theta^{-1}\rbrace$ minimizes it on $\left(-2\Theta^{-1},+\infty\right)$. Thus, the global minimizer over $\left( \dfrac{-2\xi_{\max}^{-}}{1-(\xi_{\max}^{-})^{2}}, +\infty \right)$ is $\alpha_{T}^{*}=\mathrm{max}\lbrace \left(\xi_{\max}^{+}\right)^{-1},~-2\Theta^{-1}\rbrace$.
    \item[(b)] If $\xi_{\max}^{+}> 1$, then $\frac{-2\xi_{\max}^{-}}{1-(\xi_{\max}^{-})^{2}}<\alpha<\frac{2\xi_{\max}^{+}}{(\xi_{\max}^{+})^{2}-1}$. We have 
    \begin{align*}
        \rho\left(G\left(T;\alpha\right)\right)=
        \begin{cases}
            \Xi_{1}\left(\alpha\right),~\dfrac{-2\xi_{\max}^{-}}{1-(\xi_{\max}^{-})^{2}}<\alpha\le-2\Theta^{-1},\\[10pt]
            \Xi_{2}\left(\alpha\right),~-2\Theta^{-1}<\alpha<\dfrac{2\xi_{\max}^{+}}{(\xi_{\max}^{+})^{2}-1}.
        \end{cases}
    \end{align*}
    Analogously, $\alpha=\mathrm{max}\lbrace \left(\xi_{\max}^{+}\right)^{-1},~-2\Theta^{-1}\rbrace$ minimizes $\rho\left(G\left(T;\alpha\right)\right)$ in the interval $\left(\dfrac{-2\xi_{\max}^{-}}{1-(\xi_{\max}^{-})^{2}},\dfrac{2\xi_{\max}^{+}}{(\xi_{\max}^{+})^{2}-1}\right)$.
    \end{itemize}
    \end{itemize}
\end{proof}

\section{PLHSS Preconditioners and Preconditioned Krylov Subspace Methods}\label{sec:4}
The PLHSS iterative method (\ref{equ:Piter}) can be derived by matrix splitting   
\begin{align*}
	A = M\left(V;\alpha\right)-N\left(V;\alpha\right),
\end{align*}
where
\begin{align}\label{equ:preconditioner}
	M\left(V;\alpha\right)=\mathrm{i}T+\frac{\mathrm{i}}{\alpha}WV^{-1}T,
\end{align}
and
\begin{align*}
	N\left(V;\alpha\right)=-W+\frac{\mathrm{i}}{\alpha}WV^{-1}T.
\end{align*}
In addition, the splitting matrix $M\left(V;\alpha\right)$ can  serve as a preconditioner for solving the complex symmetric system (\ref{equ:system}), which is referred as the PLHSS preconditioner.  According to the analysis in the last section,  $W$ and $T$ can serve for $V$ in  $M\left(V;\alpha\right)$. Thereby, we  obtain two  preconditioners as follows
\begin{align}\label{equ:preconditioner,LHSS,W}
	\mathcal{P}_{PLW}:=M(W;\alpha)=\mathrm{i}\frac{\alpha+1}{\alpha}T,
\end{align}
\begin{align}\label{equ:preconditioner,LHSS,T}
	\mathcal{P}_{PLT}:=M(T;\alpha)=\mathrm{i}(T+\frac{\mathrm{1}}{\alpha}W).   
\end{align}
To gauge the performance of the above  two preconditioners, we  analyze the  eigenvalues distribution and eigenvectors of the preconditioned matrices as follows.
\begin{thm}
	Let $A=W+\mathrm{i}T$, with symmetric positive definite matrix $W\in\mathbb{R}^{n\times n}$ and symmetric indefinite matrix $T\in\mathbb{R}^{n\times n}$. Let $\alpha$ be a positive constant. Define $Z=W^{\frac{1}{2}}T^{-1}W^{\frac{1}{2}}$. Denote by $\xi_{1},~\xi_{2},\cdots,\xi_{n}$ the eigenvalues of the symmetric matrix $Z\in\mathbb{R}^{n\times n}$, and by $q_{1},~q_{2},\cdots,q_{n}$ the corresponding orthogonal eigenvectors. Then the eigenvalues of the matrix $\mathcal{P}_{PLW}^{-1}A$ are given by
	\begin{align*}
		\eta_{W,j}=\frac{\alpha}{\alpha+1}\left(1-\mathrm{i}\xi_{j}\right),~j=1,2,\cdots,n,
	\end{align*}
	and the corresponding eigenvectors are given by
	\begin{align*}
		x_{W,j} = W^{-\frac{1}{2}}q_{j},~j=1,2,\cdots,n.
	\end{align*}
	Therefore, it holds that $\mathcal{P}_{PLW}^{-1}A=X_{W}\Lambda_{W}X_{W}^{-1}$, where $X_{W}=(x_{W,1},x_{W,2},\cdots,x_{W,n})\in\mathbb{R}^{n\times n}$ and $\Lambda_{W}=(\eta_{W,1},\eta_{W,2},\cdots,\eta_{W,n})\in\mathbb{C}^{n\times n}$, with $\kappa_{2}\left(X_{W}\right)=\sqrt{\kappa_{2}\left(W\right)}$.
\end{thm}

\begin{thm}\label{thm:PPLT}
	Let $A=W+\mathrm{i}T$, with symmetric positive definite matrix $W\in\mathbb{R}^{n\times n}$ and symmetric indefinite matrix $T\in\mathbb{R}^{n\times n}$. Let $\alpha$ be a positive constant such that the matrix $\alpha T+W$ is positive definite. Define $Z^{(\alpha)}=\left(\alpha T+W\right)^{-\frac{1}{2}}\left(T-\alpha W\right)\left(\alpha T+W\right)^{-\frac{1}{2}}$. Denote by $\tau_{1},~\tau_{2},\cdots,\tau_{n}$ the eigenvalues of the symmetric matrix $Z^{(\alpha)}\in\mathbb{R}^{n\times n}$, and by $p_{1},~p_{2},\cdots,p_{n}$ the corresponding orthogonal eigenvectors. Then the eigenvalues of the matrix $\mathcal{P}_{PLT}^{-1}A$ are given by
	\begin{align*}
		\eta_{T,j}=\frac{\alpha}{\alpha+\mathrm{i}}\left(1+\mathrm{i}\tau_{j}\right),~j=1,2,\cdots,n,
	\end{align*}
	and the corresponding eigenvectors are given by
	\begin{align*}
		x_{T,j} = \left(\alpha T+W\right)^{-\frac{1}{2}}p_{j},~j=1,2,\cdots,n.
	\end{align*}
	Therefore, it holds that $\mathcal{P}_{PLT}^{-1}A=X_{T}\Lambda_{T}X_{T}^{-1}$, where $X_{T}=(x_{T,1},x_{T,2},\cdots,x_{T,n})\in\mathbb{R}^{n\times n}$ and $\Lambda_{T}=(\eta_{T,1},\eta_{T,2},\cdots,\eta_{T,n})\in\mathbb{C}^{n\times n}$, with $\kappa_{2}\left(X_{T}\right)=\sqrt{\kappa_{2}\left(\alpha T+W\right)}$.
\end{thm}
The proofs of the above two theorems can be found in the appendix.

\begin{remark}
	Actually, we could also deduce the relationship between $\xi_{j}$ and $\tau_{j}$ through straightforward calculation. According to the definition of eigenvalues and eigenvectors, 
	\begin{align*}
		\left(\alpha T+W\right)^{-1}\left(T-\alpha W\right)\left(\alpha T+W\right)^{-\frac{1}{2}}p_{j} = \tau_{j}\left(\alpha T+W\right)^{-\frac{1}{2}}p_{j}.
	\end{align*}
	Then, with the notation $\tilde{p}_{j}=\left(\alpha T+W\right)^{-\frac{1}{2}}p_{j}$
	\begin{align*}
		\left(1-\alpha \tau_{j}\right) T \tilde{p}_{j} = \left(\alpha+\tau_{j}\right) W\tilde{p}_{j},
	\end{align*}
	which implies that $\frac{1-\alpha \tau_{j}}{\alpha+\tau_{j}}$ is expressed as the eigenvalues of $T^{-1}W$. $T^{-1}W$ has the same eigenvalues with the matrix $Z$ denoted in Theorem 4.1, so we conclude that
	\begin{align*}
		\xi_{j} = \frac{1-\alpha \tau_{j}}{\alpha+\tau_{j}},
	\end{align*}
	equally,
	\begin{align*}
		\tau_{j} = \frac{1-\alpha \xi_{j}}{\alpha + \xi_{j}}.
	\end{align*}
	Therefore, the eigenvalues of $\mathcal{P}_{PLT}^{-1}A$ could be written as 
	\begin{align*}
		\eta_{T,j}=\frac{\alpha}{\alpha+\xi_{j}}\left(1-\mathrm{i}\xi_{j}\right),~j=1,2,\cdots,n.
	\end{align*}
\end{remark}

\begin{remark}
	In fact, $\alpha\in\left(0,-\dfrac{\lambda_{\min}}{\mu_{1}}\right)$ can reasonably serve as the parameter selection in Theorem \ref{thm:PPLT}. Denote any eigenvalue of $\alpha T +W$ as $\lambda$, We have that
	\begin{align*}
		0 =  -\frac{\lambda_{\min}}{\mu_{1}}\mu_{1} + \lambda_{\min} \le \alpha \mu_{1} +\lambda_{\min} \le 
		\lambda.
	\end{align*}	
	This indicates that $\alpha T+W$ is positive definite, and thus $\left(\alpha T+W\right)^{\frac{1}{2}}$ exists.
\end{remark}

\begin{remark}
	\cite{BaiBook} It is well-known that the convergence rate of GMRES is intrinsically related to the eigenvector conditioning of the coefficient matrix.  Denoted by $x_{0}$ the initial guess, $r_{0} = b-Ax_{0}$ is the initial residual. Let $\mathcal{P}_{k}$ be the  set of polynomials with degree at most k.  For a diagonalizable matrix $A = X\Lambda X^{-1}$, with 
    $\Lambda = \mathrm{diag}(\lambda_{1},\cdots, \lambda_{n})$,   the k-th step iterative residual of the GMRES method satisfies
	\begin{align*}
		\frac{\lVert b-A x^{(k)}\rVert_{2}}{\lVert r_{0}\rVert_{2}}\le \kappa_{2}\left(X\right) \mathop{min}_{q\in\mathcal{P}_{k},q\left(0\right)=1}\mathop{max}_{1\le i\le n}|q(\lambda_{i})|,
	\end{align*}
	where $\kappa_{2}(X) = \Vert X \Vert_{2}\Vert X^{-1} \Vert_{2}$ is the condition number of the eigenvector matrix. Consequently, it is important to enhance the performance of Krylov subspace iteration method by improving eigenvector conditioning as well as clustering eigenvalues of preconditioned matrix. 
\end{remark}

\section{Numerical examples}\label{sec:5}

In this section, we test the performance of the proposed LHSS and PLHSS iterative methods through numerical experiments, comparing them with existing approaches using model problems derived from surface acoustic wave device simulations.    We further evaluated how does the PLHSS preconditioners work with GMRES and COCG iterative solvers. The numerical experiments are carried out using MATLAB R2024b on Intel(R) Core(TM) i7-8550U CPU @ 1.80GHz with Ubuntu 22.01.5 LTS operating system. In our implementations, the initial guess is chosen to be $x^{(0)}=0$, the iterative methods are carried out within $k_{\max}=500$ iterations until the relative residual reaches stopping criterion
\begin{align*}
	\mathrm{RES}:=\frac{\Vert b-A x\Vert_{2}}{\Vert b\Vert_{2}}\le 10^{-8}.
\end{align*}
The computational efficiency is quantified by the number of iterations ({\bf IT}) and the elapsed CPU time ({\bf CPU}). We denote the degrees of freedom of the linear system of equations as {\bf DOF}.

\begin{example}\label{exp:piezoelectric} In SAW simulation, the piezoelectric equation
	\begin{align*}
		\begin{cases}
			\nabla_{s}^{T}\left(c\nabla_{s}\mathbf{u}\right)-\omega^{2}\rho\mathbf{u}+\nabla_{s}^{T}\left(e^{T}\nabla\phi\right)=0\\
			\nabla\cdot\left(e\nabla_{s}\mathbf{u}\right)-\nabla\cdot\left(\varepsilon \nabla\phi\right)=0
		\end{cases}
	\end{align*}
is discretized by the standard finite element method.  The  linear system of equations
	\begin{align}\label{sys:ind}
		\left(K-\omega^{2}M\right)x=b,
	\end{align}
emerges,  where the stiffness matrix $K\in\mathbb{C}^{N\times N}$ is complex symmetric, and the mass matrix 
	 $M\in\mathbb{R}^{N\times N}$ is symmetric definite positive.  $\omega$ denotes angular frequency. The  indefiniteness of the coefficient matrix in (\ref{sys:ind}) intensifies as omega increases. 	
\end{example}
The model problem is discretized in various mesh sizes, leading to linear systems with different \textbf{DOF}.  The linear system  (\ref{exp:piezoelectric}) is solved by using LHSS iteration with parameters 
$\alpha=\alpha^{*},1$ as well as the PLHSS iteration methods (\ref{equ:PiterW}) and (\ref{equ:PiterT}) with  $\alpha=\alpha_{W}^{*},\alpha_{T}^{*},1$ respectively.  
Next, we verify the correctness of the optimal iterative parameter theory presented in the Remark \ref{rmk:optalpW} and Theorem \ref{thm:optalphaT}. In  Figure \ref{fig:Iter_methods}, we plot the number of iterations required for the iterative methods to meet the convergence condition with different parameters. Here, a smaller number of iterations indicates that the chosen $\alpha$  performs better. We label the theoretical optimal parameter $\alpha_{W}^{*}$ and $\alpha_{T}^{*}$ calculated according to Remark \ref{rmk:optalpW} and Theorem \ref{thm:optalphaT} with red stars.
We notice that the iterations remain virtually insensitive within a broad neighborhood of the optimal parameter. Numerical experiments demonstrate that the PLHSS iteration method exhibits robust performance with respect to parameter variations, maintaining stable convergence rates across a wide range of parameter values.
\begin{figure}[htbp]
\raggedright
\subfigure[Iterations with $V=W$ when $N=74124$.]
{
	\begin{minipage}[h]{0.45\linewidth}
		\centering
		\includegraphics[width=2.7in]{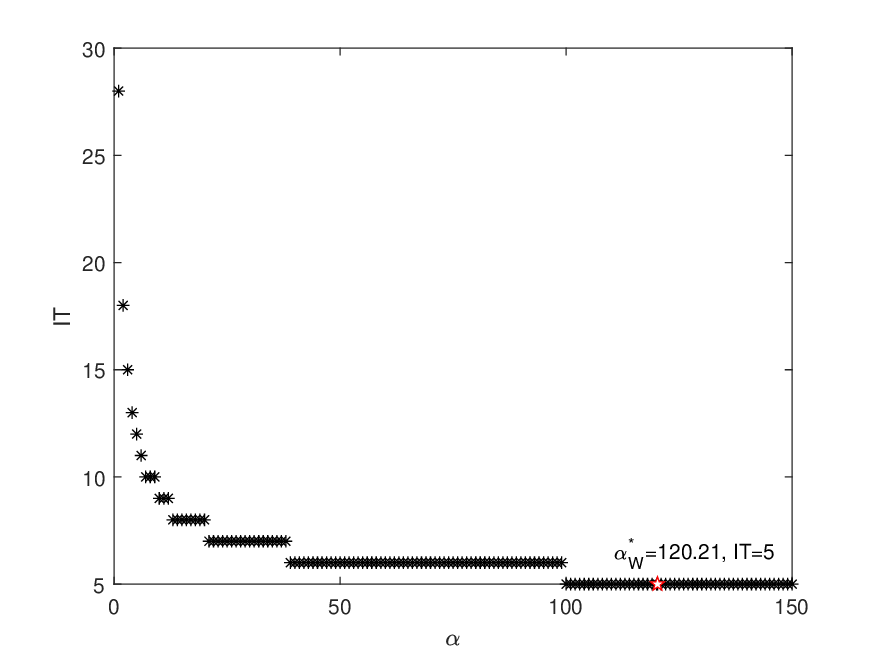}\label{fig:IT-W-74124}
\end{minipage}}
\subfigure[Iterations with $V=W$ when $N=130500$.]{
	\begin{minipage}[h]{0.45\linewidth}
		\centering
		\includegraphics[width=2.7in]{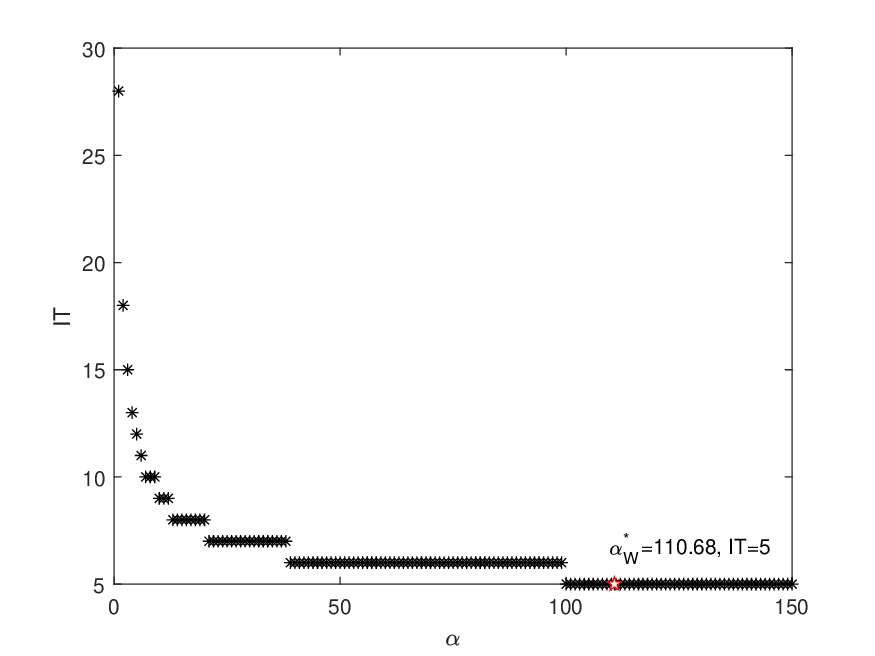}\label{fig:IT-W-130500}
	\end{minipage}}
\subfigure[Iterations with $V=T$ when $N=74124$.]
{
	\begin{minipage}[h]{0.45\linewidth}
		\centering
		\includegraphics[width=2.7in]{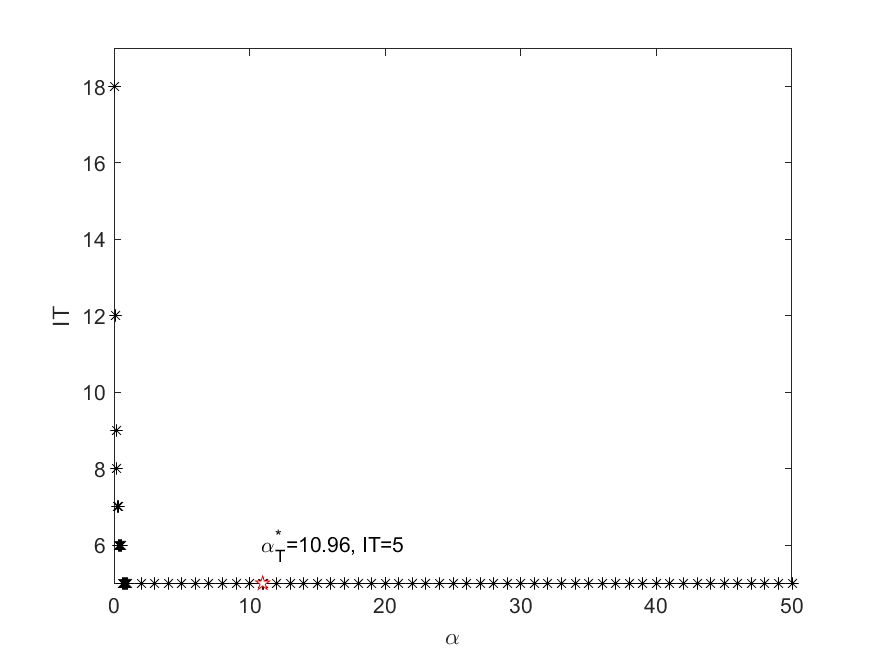}\label{fig:IT-T-74124}
\end{minipage}}
\subfigure[Iterations with $V=T$ when $N=130500$.]{
	\begin{minipage}[h]{0.45\linewidth}
		\centering
		\includegraphics[width=2.7in]{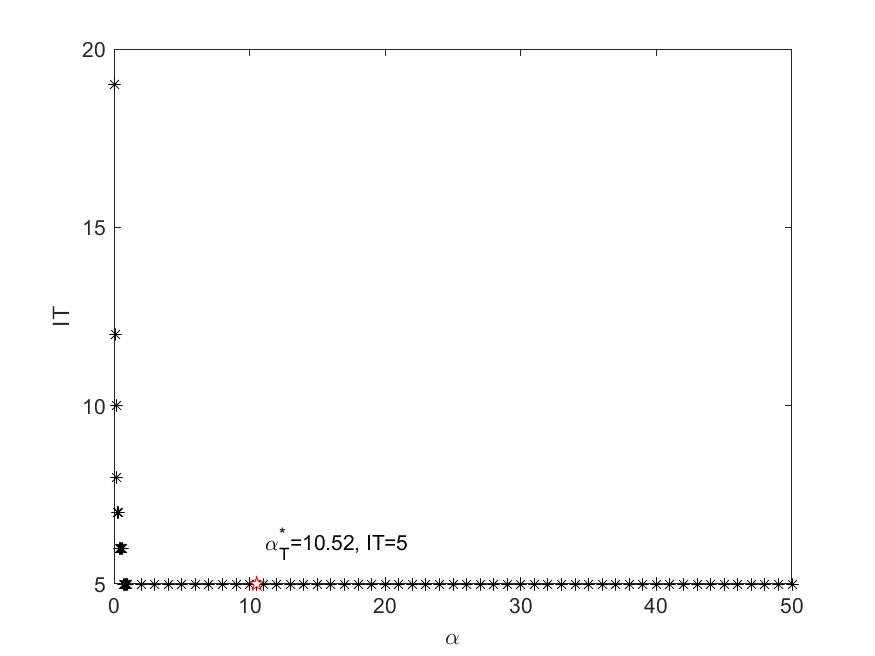}\label{fig:IT-T-130500}
	\end{minipage}
}
\caption{Numbers of iterations versus variable values of $\alpha$ for PLHSS iteration method  when $\omega=8\times 10^{9}$.}\label{fig:Iter_methods}
\end{figure}

We present in Table \ref{TAB:Iterative-LHSS} the number of iterations and CPU time required for solving (\ref{exp:piezoelectric}) using  LHSS, PLHSS,  PMHSS \cite{PMHSS}, HSS \cite{HSS}, and LPMHSS \cite{LPMHSS} iteration methods with variable parameters. It should be noted that although the LHSS has theoretical convergence guarantees for solving indefinite linear systems, when the system is ill conditioned, the  parameter $\alpha$  needs to be chosen as small as in the order of $10^{-9}$, and the convergence rate becomes sluggish, failing to achieve a qualified approximate solution within a limited number of iterations. This also highlights the necessity of the PLHSS iteration method. As the angular frequency $\omega$ increases, the indefiniteness of the matrix $K-\omega^{2}M$ progressively intensifies. Nevertheless, the PLHSS solver maintains stable iteration counts and CPU times when solving (\ref{exp:piezoelectric}), demonstrating remarkable robustness against increasing indefiniteness. When employing $T$ as a preconditioner, the PLHSS iterative method achieves convergence within 5 iterations across all test cases. 
It should be noted that the HSS method theoretically converges when solving linear systems with a positive definite Hermitian part, whereas PMHSS and LPMHSS lack theoretical convergence guarantees. Nevertheless, PMHSS still converges in this example. However, both HSS and LPMHSS fail to achieve convergence within the specified number of iterations. Therefore, we don't list the numerical data of them in Table \ref{TAB:Iterative-LHSS}.   In Figure \ref{fig:Iter_numbers}, we plot the curves of relative residuals versus iteration steps for various methods.
The PLHSS iteration methods converge significantly faster than the PMHSS iteration method. The relative residual of the HSS iteration method decreases, but at a very slow rate. In contrast, the relative residual of the LPMHSS iteration method increases. LPMHSS iteration method diverges for this indefinite system.
\begin{table}[htbp]
	\centering
	\caption{IT and CPU(in parentheses) of LHSS,PLHSS and PMHSS iteration methods. }
	\label{TAB:Iterative-LHSS}
	\scriptsize
	\begin{tabular}{cccllll}
		\toprule
	{\multirow{2}{*}{DOF}}&{\multirow{2}{*}{Method}}&&\multicolumn{1}{l}{$\omega$}\\ \cline{4-7}
	&&&$4\times 10^8$&$8\times 10^8$&$4\times 10^{9}$&$8\times 10^{9}$\\ \hline
	{\multirow{7}{*}{74124}}	&	{\multirow{2}{*}{LHSS}}&	$\alpha^{*}$ & $1.01e-8$ & $6.12e-9$ & $9.32e-8$ &	$8.85e-7$ \\
    &&& 500(-) & 500(-)	&	500(-)	&	500(-)\\ \cline{2-7}
	&	{\multirow{2}{*}{PLHSS}}&	$\alpha_{W}^{*}$	&	$6.05e+4$	&$7.05e+2$ &$2.67e+1$ & $1.20e+2$ 	\\	
	&{\multirow{2}{*}{$V=W$}}	&	  &4(4.92)	&4(4.79)	&7(5.56) &5(5.16)  \\ \cline{3-7}
	&	&	 $\alpha=1$ &28(11.24)	&28(11.45)	&28(11.27) &28(12.93)  \\	\cline{2-7}
	&		{\multirow{2}{*}{PLHSS}}&	$\alpha_{T}^{*}$	& $2.46e+2$ & $2.66e+1$ & $5.17$ & $1.10e+1$ \\	
	&  {\multirow{2}{*}{$V=T$}} &	  & 4(10.67)&4(10.89)	&5(11.89) &5(11.44) \\	\cline{3-7}
	&	&	 $\alpha=1$ & 4(11.16)&4(10.94)	&5(11.53) &5(11.54) \\ \cline{2-7}	
   &    PMHSS&  $\alpha_{opt}=1$	& 55(41.85) &55(43.98)&55(41.21)&55(55.90)	\\ 
		\bottomrule
	{\multirow{7}{*}{130500}}	&	{\multirow{2}{*}{LHSS}}&	$\alpha^{*}$	&	$4.31e-9$ &	$2.38e-9$	&	$1.57e-7$	&	$3.32e-6$	\\
    &&&	500(-) &	500(-)	&	500(-)	&	500(-)	\\	\cline{2-7}
	&	{\multirow{2}{*}{PLHSS}}&	$\alpha_{W}^{*}$	&	$6.09e+4$&$8.09e+2$ &$1.29e+2$ &$1.11e+2$  \\	
	&	{\multirow{2}{*}{$V=W$}}&	 &4(8.59)&4(8.42)	&5(9.47) &5(9.35) \\ \cline{3-7}
	&	&	 $\alpha=1$ &28(19.11)&28(20.15)	&28(20.29) &28(20.16) \\	\cline{2-7}
	&	{\multirow{2}{*}{PLHSS}}&	$\alpha_{T}^{*}$ & $2.47e+2$ & $2.84e+1$	& $1.14e+1$ & $1.05e+1$ \\	
	&{\multirow{2}{*}{$V=T$}}	&	 & 4(18.17)&4(18.00)	&4(18.00) &5(18.89) \\	\cline{3-7}
	&	&	 $\alpha=1$ & 4(18.05)&4(17.96)	&4(17.88) &5(18.75) \\	\cline{2-7}	
      &    PMHSS&  $\alpha_{opt}=1$	&55(64.45) &55(63.60)&55(65.54)&55(63.79)\\
		\bottomrule
	\end{tabular}
\end{table}

\begin{figure}[htbp]
\raggedright
\subfigure[Relative residual when $N=74124$.]
{
	\begin{minipage}[h]{0.45\linewidth}
		\centering
		\includegraphics[width=2.7in]{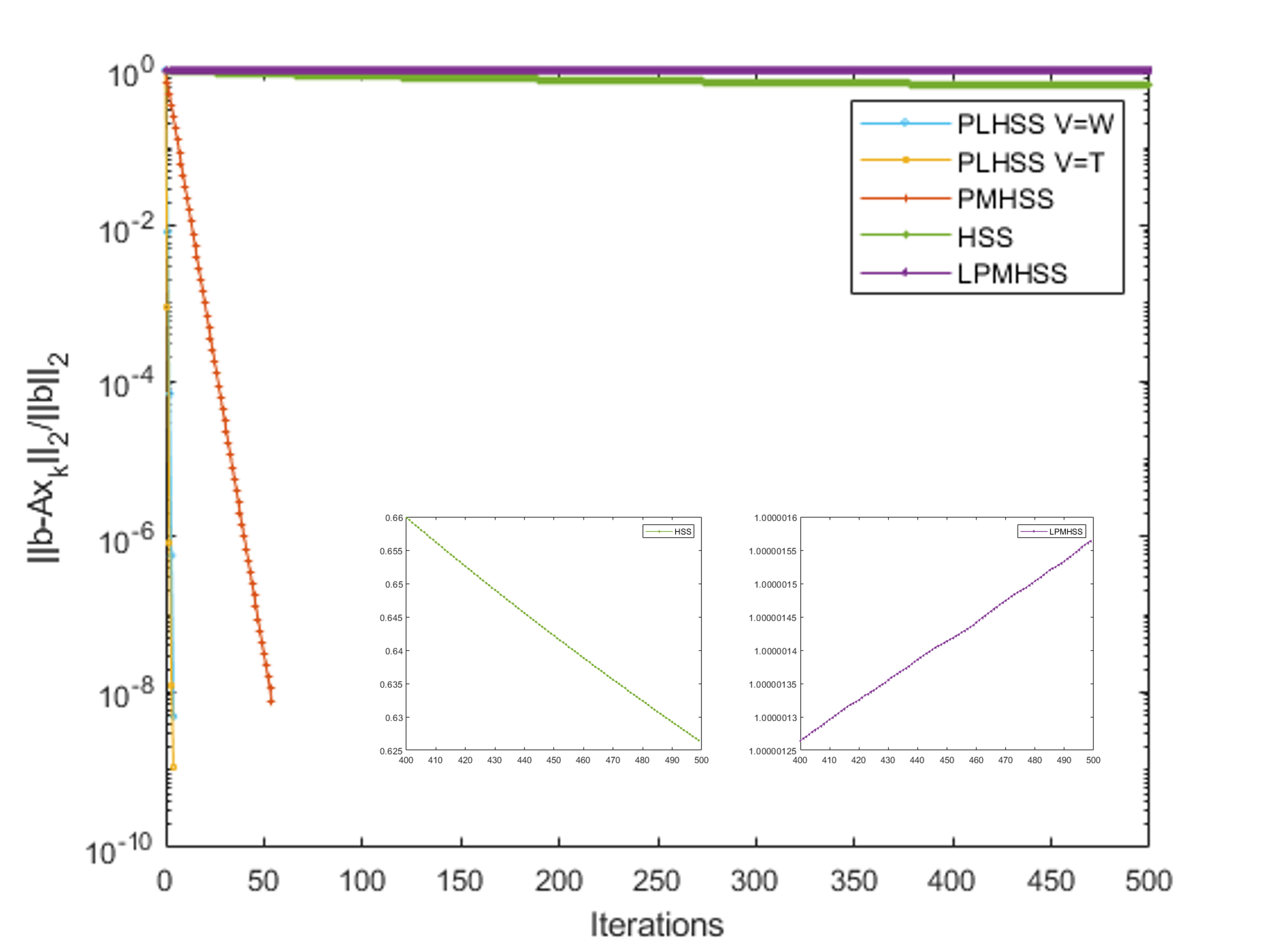}\label{fig:Iter74124}
\end{minipage}}
\subfigure[Relative residual when $N=130500$.]{
	\begin{minipage}[h]{0.45\linewidth}
		\centering
		\includegraphics[width=2.7in]{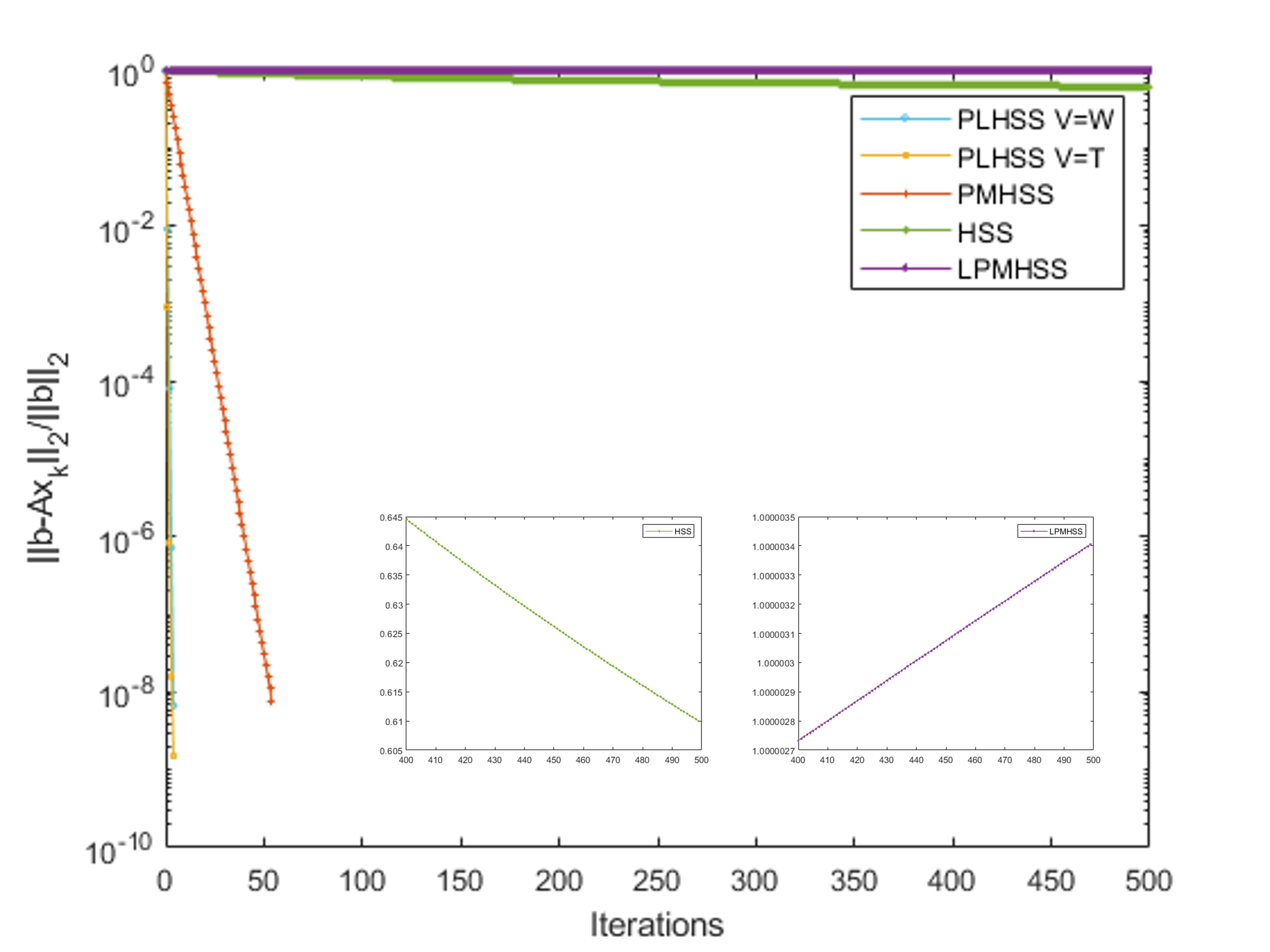}\label{fig:Iter130500}
	\end{minipage}
}
\caption{Relative residual versus number of iterations for different iteration methods  when $\omega=8\times 10^{9}$.}\label{fig:Iter_numbers}
\end{figure}

Next, we solve (\ref{exp:piezoelectric}) using the GMRES method with preconditioners 
$\mathcal{P}_{PLW}$ and $\mathcal{P}_{PLT}$, and compare their performance with several established preconditioners: the well-known HSS preconditioner  \cite{HSSpre}, the PMHSS preconditioner  \cite{PMHSS}, the LPMHSS preconditioner 
 \cite{LPMHSS}, and the C-to-R preconditioner  \cite{CtoR}
 \begin{equation}\label{preconditioner:CtoR}
     \mathcal{P}_{CtoR} =
     \begin{bmatrix}
         W & T\\
         -T & W+2T
     \end{bmatrix}.
 \end{equation}
 The  generalized residual equations in preconditioning process are solved by implementing  $LDL^{T}$ factorizations.  The number of iterations  and CPU time of the GMRES methods with different preconditioners  are listed in  Table \ref{TAB:Precond-Comparision}. It is worth noting that both the PMHSS preconditioner and the LPMHSS preconditioner are typically applied to linear systems with positive definite  non-Hermitian part, with little research exploring their application when non-Hermitian part is  indefinite. However, we have found that these preconditioners differ from the PLHSS preconditioner only by a constant factor.  This explains why their numerical performance is very similar when used as preconditioners.  Except for HSS-preconditioned GMRES,  discretizing mesh size and angular frequency have minimal impact on the convergence of other preconditioned GMRES solvers.  While the PLHSS preconditioned GMRES method achieves convergence with the fewest iterations, the C-to-R preconditioned GMRES method wins in terms of speed, reaching convergence faster. The difference between their convergence rates is marginal.  Unlike the nonsymmetric  $\mathcal{P}_{CtoR} $ in (\ref{preconditioner:CtoR}), the PLHSS preconditioners $\mathcal{P}_{PLW}$ and $\mathcal{P}_{PLT}$  are   multiple  of a symmetric positive definite matrix. Therefore, they can serve as the preconditioners for the  COCG method, which is the three-term recurrence Krylov subspace method for  solving the complex symmetric linear system of equations.

\begin{table}[htbp]
	\centering
	\caption{IT and CPU(in parentheses) of preconditioned GMRES method with different preconditioners.}
	\label{TAB:Precond-Comparision}
	\footnotesize
	\begin{tabular}{clllllll}
		\toprule
		{\multirow{2}{*}{DOF}}&{\multirow{2}{*}{Precond.}}&\multicolumn{1}{l}{$\omega$}\\ \cline{3-8}
		&&$4\times 10^8$&$8\times 10^8$&$4\times 10^{9}$&$8\times 10^{9}$&$4\times 10^{10}$&$8\times 10^{10}$\\ \hline
		{\multirow{6}{*}{74124}} 
		&$\mathcal{P}_{HSS}$& 401(683.37)&359(568.92) &  280(430.21)  & 244(331.11) & 208(317.34) & 195(286.78)\\
		&$\mathcal{P}_{PMHSS}$&  4(5.94)& 5(7.06)& 6(7.36)  & 6(7.38) & 15(11.36) & 14(10.93)\\
		&$\mathcal{P}_{LPMHSS}$&   4(5.83) & 5(6.74)&  6(7.12) & 6(7.22) & 15(11.09) & 14(10.62) \\
		&$\mathcal{P}_{CtoR}$&   3(4.22)& 4(5.19) &  5(5.38) & 4(5.04) & 8(5.93) & 7(5.71) 	\\
		&$\mathcal{P}_{PLW}$&  3(5.27)& 4(6.17)&  4(6.14) & 4(6.20) & 7(7.50) & 7(7.51) \\
		&$\mathcal{P}_{PLT}$&  3(5.32)& 4(6.29) & 4(6.10)& 4(6.20) & 7(7.40) & 7(7.48) \\\hline
		{\multirow{5}{*}{130500}} 
		&$\mathcal{P}_{HSS}$& 500(-)&	500(-) & 500(-)& 500(-) &500(-) & 500(-)\\
		&$\mathcal{P}_{PMHSS}$& 4(11.74)&5(12.74) & 6(13.59) & 6(13.32) & 14(19.89) & 18(23.01) \\
		&$\mathcal{P}_{LPMHSS}$& 4(11.90)&5(11.69) & 6(13.48) & 6(13.45) & 14(19.83) & 18(23.85) \\
		&$\mathcal{P}_{CtoR}$& 3(9.10)& 4(9.78) & 5(10.11) & 8(11.33) & 8(11.33) & 10(11.81) \\
		&$\mathcal{P}_{PLW}$& 3(11.10)& 4(11.81) & 4(12.02) & 7(14.13) & 7(14.13) & 8(14.58) \\
		&$\mathcal{P}_{PLT}$& 3(11.10)& 4(11.86) & 4(11.84) & 7(14.45) & 7(14.45) & 8(14.59) \\
		\bottomrule
	\end{tabular}
\end{table}
We adopt the short-recurrence method COCG \cite{COCG} to a system with $\mathbf{DOF}=208170$ to evaluate the applicability of the preconditioners. In Table \ref{TAB:GMRES-COCG}, iterations and CPU time of PLHSS preconditioned COCG method  and C-to-R preconditioned GMRES method are listed.  The C-to-R preconditioner exhibits a residual stagnation, which may prevent it from achieving the desired accuracy. Upon examining the code, it turns out that the orthogonality of the basis functions of the Krylov subspace has been compromised. In figure \ref{fig:Compare}, we plot the curves of relative residuals versus iteration steps for PLHSS-preconditioned COCG method and GMRES method preconditioned with $\mathcal{P}_{CtoR}$ when $N=208170$. We observe that the PLHSS preconditioner is also effective in short-recurrence Krylov subspace methods.
\begin{table}[htbp]
	\centering
	\caption{IT and CPU(in parentheses) for solving (\ref{exp:piezoelectric}) by preconditioned GMRES method with different preconditioners.}
	\label{TAB:GMRES-COCG}
	\footnotesize
	\begin{tabular}{cllll|ll}
		\toprule
·		{\multirow{2}{*}{DOF}}&{\multirow{2}{*}{Method}}&{\multirow{2}{*}{Precond.}}& \multicolumn{2}{l}{$\text{tol}=10^{-8}$} &\multicolumn{2}{l}{$\text{tol}=10^{-10}$} \\ \cline{4-7}
		&&&$\omega=8\times 10^8$&$\omega=8\times 10^{10}$&$\omega=8\times 10^8$&$\omega=8\times 10^{10}$\\ \hline
		{\multirow{3}{*}{208170}} 
		&GMRES&$\mathcal{P}_{CtoR}$& 4(21.82)& 12(27.39)&  500(-) & 500(-)	\\
        &{\multirow{2}{*}{COCG}}&$\mathcal{P}_{PLW}$& 3(18.13)& 10(27.33) &  3(18.83) & 15(32.14) 	\\
		&&$\mathcal{P}_{PLT}$& 3(19.89)& 11(27.14) &  3(19.28) & 15(32.30) 	\\
		\bottomrule
	\end{tabular}
\end{table}

\begin{figure}[htbp]
\raggedright
\subfigure[$\text{tol}=10^{-8}$, $\omega=8\times10^{8}$.]
{
	\begin{minipage}[h]{0.45\linewidth}
		\centering
		\includegraphics[width=2.7in]{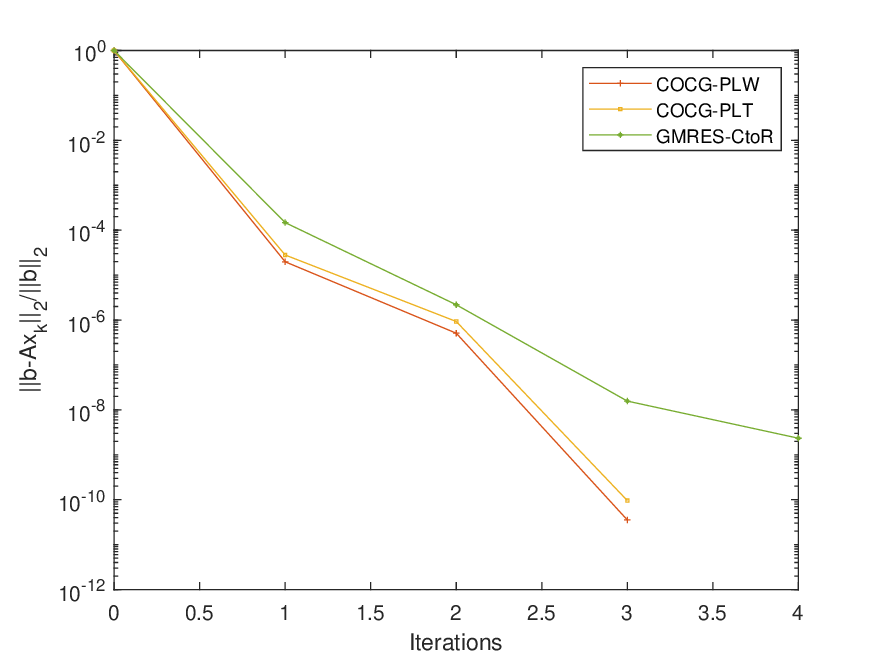}\label{fig:1e8_w1}
\end{minipage}}
\subfigure[$\text{tol}=10^{-8}$, $\omega=8\times10^{10}$.]{
	\begin{minipage}[h]{0.45\linewidth}
		\centering
		\includegraphics[width=2.7in]{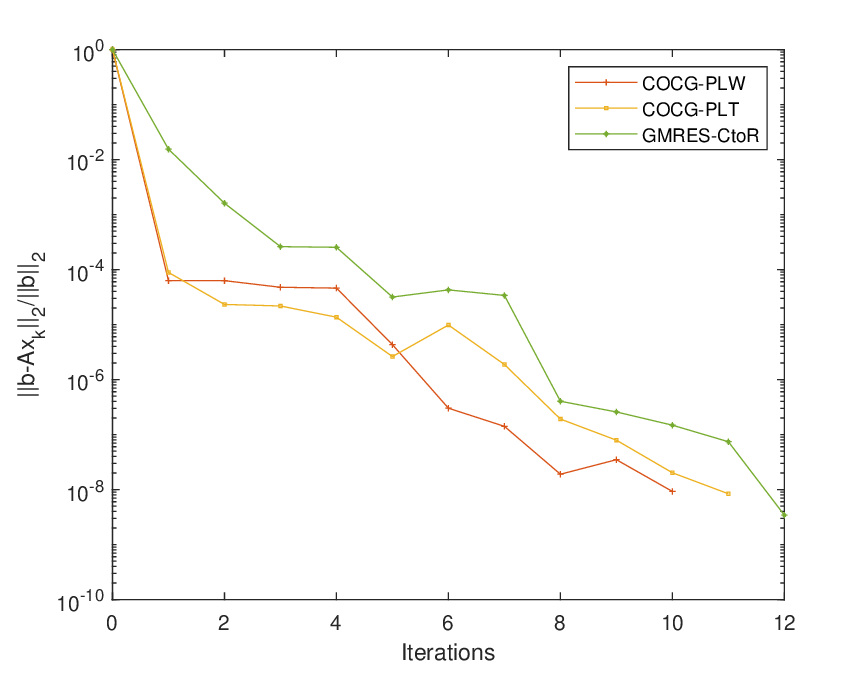}\label{fig:1e8_w2}
	\end{minipage}}
\subfigure[$\text{tol}=10^{-10}$, $\omega=8\times10^{8}$.]
{
	\begin{minipage}[h]{0.45\linewidth}
		\centering
		\includegraphics[width=2.7in]{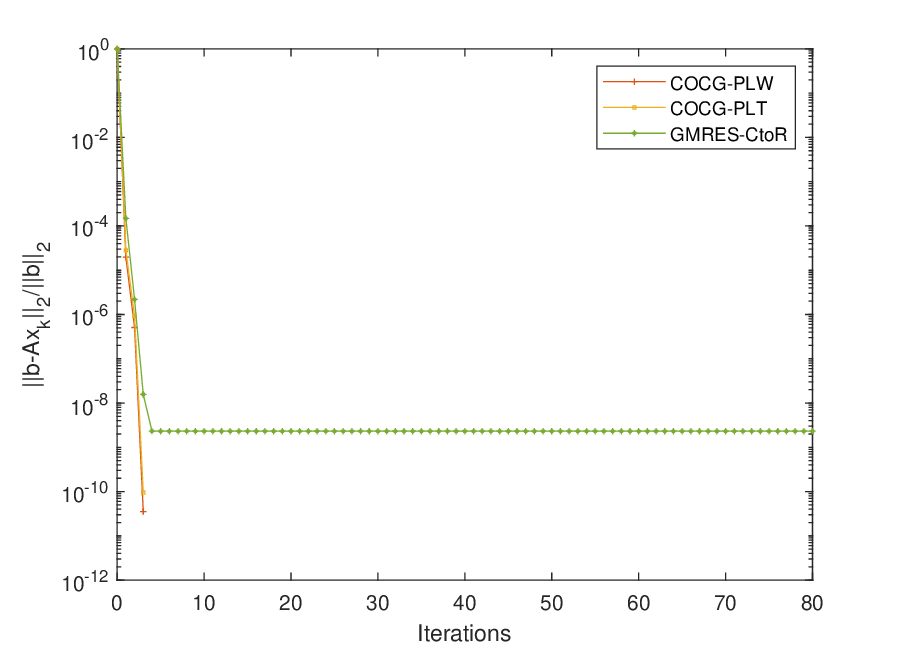}\label{fig:1e10_w1}
\end{minipage}}
\subfigure[$\text{tol}=10^{-10}$, $\omega=8\times10^{10}$.]{
	\begin{minipage}[h]{0.45\linewidth}
		\centering
		\includegraphics[width=2.7in]{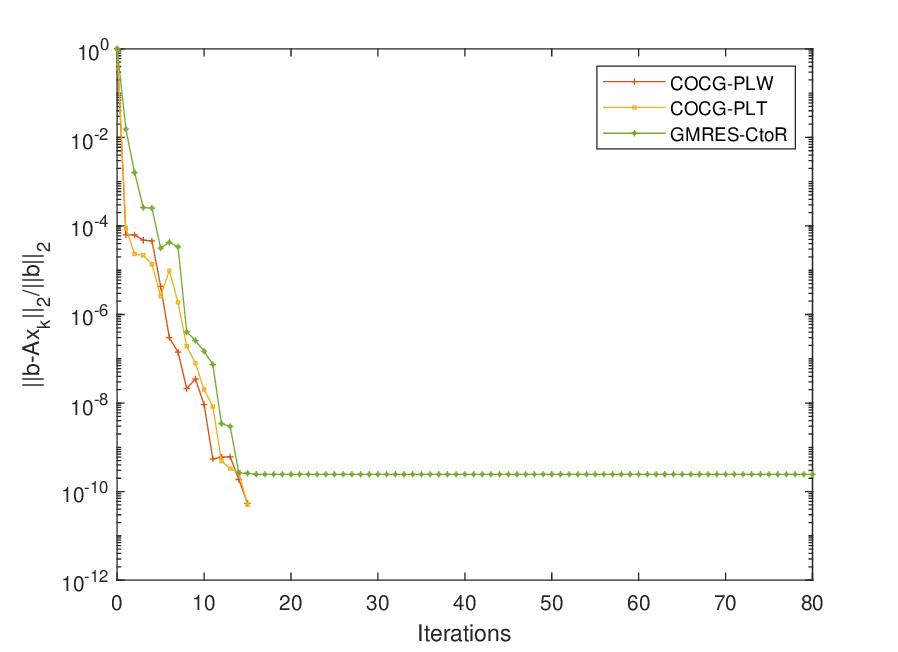}\label{fig:1e10_w2}
	\end{minipage}
}
\caption{Relative residual versus number of iterations for different iterative solvers when $N=208170$.}\label{fig:Compare}
\end{figure}

\section{Conclusion}

This study introduces a type of splitting iterative methods  for solving the  complex symmetric linear system of equations, along with its preconditioned Krylov subspace variants. Theoretical analysis guarantees  that the methods is  convergent  when the imaginary part of the complex symmetric coefficient matrix is indefinite—a challenging scenario. Numerical tests reveal its  efficiency over existed splitting-based iterative methods.  The proposed matrix splitting method introduce the  efficient preconditioners. These preconditioners, being scalar multiples of real symmetric positive matrices,  integrate with short-recurrence Krylov solvers like COCG and COCR, making them ideal for large scale equation systems. The matrix splitting iteration methods can be also serve as the inner solver for complex symmetric indefinite saddle-point problems, demonstrating practical value in engineering applications such as modeling surface acoustic wave device equations.

\bibliography{references}

\newpage
\appendix

\section{Proof of Theorem 4.1}

\begin{proof}[proof]
		Define matrices
		\begin{align*}
			Q=\left(q_{1},q_{2},\cdots,q_{n}\right)\in\mathbb{R}^{n\times n}
		\end{align*}
		and
		\begin{align*}
			D=\text{diag}\left(\xi_{1},\xi_{2},\cdots,\xi_{n}\right)\in\mathbb{R}^{n\times n}.
		\end{align*}
		Then it holds that
		\begin{align*}
			Z = Q D Q^{T}.
		\end{align*}
		By straightforward computations we have
		\begin{align*}
			\mathcal{P}_{PLW}^{-1}A &= \frac{\alpha}{\mathrm{i}\left(\alpha+1\right)}T^{-1}\left(W+\mathrm{i}T\right)\\
			&=\frac{\alpha}{\mathrm{i}\left(\alpha+1\right)\left(-\mathrm{i}\right)}\left(I-\mathrm{i}T^{-1}W\right)\\
			&=\frac{\alpha}{\alpha+1}\left(W^{-\frac{1}{2}}W^{\frac{1}{2}}-\mathrm{i}T^{-1}W^{\frac{1}{2}}W^{\frac{1}{2}}\right)\\
			&=\frac{\alpha}{\alpha+1}W^{-\frac{1}{2}}\left(I-\mathrm{i}W^{\frac{1}{2}}T^{-1}W^{\frac{1}{2}}\right)W^{\frac{1}{2}}\\
			&=W^{-\frac{1}{2}}Q\left(\frac{\alpha}{\alpha+1}\left(I-\mathrm{i}D\right)\right)Q^{T}W^{\frac{1}{2}}\\
			&=X_{W}\Lambda_{W}X_{W}^{-1}.
		\end{align*}
		
		Hence, the eigenvalues of the matrix $\mathcal{P}_{PLW}^{-1}A$ are given by
		\begin{align*}
			\eta_{W,j}=\frac{\alpha}{\alpha+1}\left(1-\mathrm{i}\xi_{j}\right),~j=1,2,\cdots,n,
		\end{align*}
		and the corresponding eigenvectors are given by $x_{W,j}=W^{-\frac{1}{2}}q_{j},~j=1,2,\cdots,n.$
		
		Besides, as $Q\in\mathbb{R}^{n\times n}$ is orthogonal, we obtain
		\begin{align*}
			\left\lVert X_{W}\right\rVert_{2} = \left\lVert W^{-\frac{1}{2}}Q\right\rVert_{2}=\left\lVert W^{-1}\right\rVert_{2}^{\frac{1}{2}}
		\end{align*}
		and
		\begin{align*}
			\left\lVert X_{W}^{-1}\right\rVert_{2} = \left\lVert Q^{T}W^{\frac{1}{2}}\right\rVert_{2}=\left\lVert W\right\rVert_{2}^{\frac{1}{2}}.
		\end{align*}
		It then follows that
		\begin{align*}
			\kappa_{2}\left(X_{W}\right)=\left\lVert X_{W}\right\rVert_{2}\left\lVert X_{W}^{-1}\right\rVert_{2}=\left\lVert W^{-1}\right\rVert_{2}^{\frac{1}{2}}\left\lVert W\right\rVert_{2}^{\frac{1}{2}}=\sqrt{\kappa_{2}\left(W\right)}.
		\end{align*}
	\end{proof}

\section{Proof of Theorem 4.2}
\begin{proof}[proof]
		Define matrices
		\begin{align*}
			P=\left(p_{1},p_{2},\cdots,p_{n}\right)\in\mathbb{R}^{n\times n}
		\end{align*}
		and
		\begin{align*}
			D^{(\alpha)}=\text{diag}\left(\tau_{1},\tau_{2},\cdots,\tau_{n}\right)\in\mathbb{R}^{n\times n}.
		\end{align*}
		Then it holds that
		\begin{align*}
			Z^{(\alpha)} = P D^{(\alpha)} P^{T}.
		\end{align*}
		By straightforward computations we have
		\begin{align*}
			\mathcal{P}_{PLT}^{-1}A &= \frac{\alpha}{\mathrm{i}}\left(\alpha T+W\right)^{-1}\left(W+\mathrm{i}T\right)\\
			&=\frac{\alpha}{\mathrm{i}\left(1-\mathrm{i}\alpha\right)}\left(I+\mathrm{i}\left(\alpha T+W\right)^{-1}\left(T-\alpha W\right)\right)\\
			&=\frac{\alpha}{\alpha+\mathrm{i}}\left(\alpha T+W\right)^{-\frac{1}{2}}\left[I+\mathrm{i}\left(\alpha T+W\right)^{-\frac{1}{2}}\left(T-\alpha W\right)\left(\alpha T+W\right)^{-\frac{1}{2}}\right]\left(\alpha T+W\right)^{\frac{1}{2}}\\
			&=\left(\alpha T+W\right)^{-\frac{1}{2}}P\left[\frac{\alpha}{\alpha+\mathrm{i}}\left(I+\mathrm{i}D^{(\alpha)}\right)\right]P^{T}\left(\alpha T+W\right)^{\frac{1}{2}}  \\
			&=X_{T}\Lambda_{T}X_{T}^{-1}.
		\end{align*}
		
		Hence, the eigenvalues of the matrix $\mathcal{P}_{PLT}^{-1}A$ are given by
		\begin{align*}
			\eta_{T,j}=\frac{\alpha}{\alpha+\mathrm{i}}\left(1+\mathrm{i}\tau_{j}\right),~j=1,2,\cdots,n,
		\end{align*}
		and the corresponding eigenvectors are given by $x_{T,j}=\left(\alpha T+W\right)^{-\frac{1}{2}}p_{j},~j=1,2,\cdots,n.$
		
		Besides, as $P\in\mathbb{R}^{n\times n}$ is orthogonal, we obtain
		\begin{align*}
			\left\lVert X_{T}\right\rVert_{2} = \left\lVert \left(\alpha T+W\right)^{-\frac{1}{2}}P\right\rVert_{2}=\left\lVert \left(\alpha T+W\right)^{-1}\right\rVert_{2}^{\frac{1}{2}}
		\end{align*}
		and
		\begin{align*}
			\left\lVert X_{T}^{-1}\right\rVert_{2} = \left\lVert P^{T}\left(\alpha T+W\right)^{\frac{1}{2}}\right\rVert_{2}=\left\lVert \alpha T+W\right\rVert_{2}^{\frac{1}{2}}.
		\end{align*}
		It then follows that
		\begin{align*}
			\kappa_{2}\left(X_{T}\right)=\left\lVert X_{T}\right\rVert_{2}\left\lVert X_{T}^{-1}\right\rVert_{2}=\left\lVert \left(\alpha T+W\right)^{-1}\right\rVert_{2}^{\frac{1}{2}}\left\lVert \left(\alpha T+W\right)\right\rVert_{2}^{\frac{1}{2}}=\sqrt{\kappa_{2}\left(\alpha T+W\right)}.
		\end{align*}
	\end{proof}
\end{document}